\documentclass{article}

\usepackage{lineno,hyperref}
\usepackage{amsmath}
\usepackage{amssymb}
\usepackage{amsthm}
\modulolinenumbers[5]
\usepackage{enumitem}
\usepackage[all]{xy}
\SelectTips{cm}{12}

\newtheorem{lemma}{Lemma}
\newtheorem*{lemma*}{Lemma}

\newtheorem{theorem}{Theorem}
\newtheorem*{theorem*}{Theorem}

\theoremstyle{definition}

\newcommand\dotminus{\mathbin{\dot{-}}}

\newcommand\dotoplus{\mathbin{\dot{\oplus}}}

\DeclareMathOperator\mat{M}

\DeclareMathOperator\elem{E}
\DeclareMathOperator\gelem{GE}
\DeclareMathOperator\kfunctor{K}

\DeclareMathOperator\slin{SL}

\DeclareMathOperator\SC{SC}
\DeclareMathOperator\Cent{C}

\DeclareMathOperator\Ker{Ker}

\newcommand\eps{\varepsilon}
\newcommand\leqt{\trianglelefteq}

\newcommand\op{{\mathrm{op}}}
\newcommand\id{{\mathrm{id}}}

\DeclareMathOperator\Aut{Aut}
\DeclareMathOperator\Image{Im}
\DeclareMathOperator\Spec{Spec}

\newcommand\fp{{\mathrm{fp}}}
\newcommand\add{{\mathrm a}}
\newcommand\mult{{\mathrm m}}

\newcommand{\up}[2]{{^{#1}\!{#2}}}

\newcommand{\Set}{\mathbf{Set}}
\newcommand{\Group}{\mathbf{Grp}}

\newcommand{\Ring}{\mathbf{CRing}}
\newcommand{\Rng}{\mathbf{CRng}}
\DeclareMathOperator\Pro{Pro}
\DeclareMathOperator\Ind{Ind}
\DeclareMathOperator\Sub{Sub}
\newcommand{\Cat}{\mathbf{Cat}}
\newcommand{\IP}{\mathbf{IP}}
\newcommand{\IPRng}{\mathbf{IPRng}}
\newcommand{\uradical}{{\mathrm R_{\mathrm u}}}

\title{
    Locally isotropic elementary groups
}

\author{
    Egor Voronetsky\thanks{
        Research is supported by the Russian Science Foundation grant 19-71-30002.
    } \\
    Chebyshev Laboratory, \\
    St. Petersburg State University, \\
    14th Line V.O., 29B, \\
    Saint Petersburg 199178 Russia \\
}

\begin{document}
\maketitle

\begin{abstract}
    We construct elementary subgroups of all reductive groups of the local isotropic rank \(\geq 2\) over rings and prove their basic properties. In particular, our results may be applied to the automorphism groups of any finitely generated projective modules over commutative unital rings of rank \(\geq 3\) at every prime ideal.
\end{abstract}


\section{Introduction}

Let \(K\) be a unital commutative ring and \(\Phi\) be a root system of rank \(\geq 2\) (reduced and irreducible). The Chevalley group
\(
    G(\Phi, K)
\) is the point group of the Chevalley --- Demazure group scheme with the root system \(\Phi\) and some choice of the weight lattice, e.g. of adjoint type or simply connected type. This group has so-called root elements
\(
    t_\alpha(x)
\) for
\(
    \alpha \in \Phi
\),
\(
    x \in K
\), and the elementary subgroup
\(
    \elem(\Phi, K) \leq G(\Phi, K)
\) is the subgroup generated by all root elements. It is well-known \cite{abe-local, abe, fuan, kopeiko, suslin-kopeiko, taddei} that the elementary subgroup is normal and perfect (if \(\Phi\) is of type \(\mathsf B_2\) or \(\mathsf G_2\), then the perfectness holds only under the additional assumption that \(K\) does not have residue fields isomorphic to \(\mathbb F_2\)). Moreover, if \(K\) is finite-dimensional, then
\(
    G(\Phi, K) / \elem(\Phi, K)
\) is solvable \cite{k1-nilp}, i.e.
\(
    \elem(\Phi, K)
\) is the largest perfect subgroup of
\(
    G(\Phi, K)
\). Finally, if
\(
    G(\Phi, K)
\) is simply connected, then Alexei Stepanov showed in \cite{stepanov} that the width in terms of root elements of a commutator
\(
    [x, g]
\) for
\(
    x \in \elem(\Phi, K)
\),
\(
    g \in G(\Phi, K)
\) is bounded uniformly on \(K\) (i.e. the bound depends only on the root system \(\Phi\)).

The elementary group is used in algebraic \(\kfunctor\)-theory, the factor-group
\(
    \kfunctor_1(\Phi, K) = G(\Phi, K) / \elem(\Phi, K)
\) is an unstable \(\kfunctor_1\)-functor. Also, classifications of normal and subnormal subgroups of
\(
    G(\Phi, K)
\) essentially use the elementary subgroup. For example,
\(
    N \leq \slin(n, K)
\) is normalized by the elementary subgroup
\(
    \elem(\mathsf A_{n - 1}, K)
\) for
\(
    n \geq 3
\) if and only if
\(
    \elem(\mathsf A_{n - 1}, K, \mathfrak a)
    \leq N
    \leq \SC(n, K, \mathfrak a)
\) for some explicitly defined groups: the relative elementary subgroup
\(
    \elem(\mathsf A_{n - 1}, K, \mathfrak a)
\) is given by generators and the full congruence subgroup
\(
    \SC(n, K, \mathfrak a)
\) is given by equations.

It is natural to generalize these results to isotropic reductive groups. Over arbitrary unital commutative rings reductive group schemes are considered in \cite{sga3} (see also \cite{conrad}). The notion of isotropic reductive groups over fields \cite{borel-tits} is generalized to local rings (more generally, semi-local rings with connected spectra) in \cite[Exp. XXVI]{sga3}. Over arbitrary rings Victor Petrov and Anastasia Stavrova \cite{petrov-stavrova} considered reductive groups with globally defined parabolic subgroups (non-trivially intersecting all simple factors of the group scheme over geometric points), they proved the normality of the elementary subgroup in this context. In \cite{luzgarev-stavrova} Alexander Luzgarev and Stavrova showed that the elementary subgroup is perfect (under a natural additional assumption). Stavrova and Stepanov also proved \cite{stavrova-stepanov} that such isotropic reductive groups admit a standard classification of subgroups normalized by the elementary group if the structure constants are invertible in \(K\). For twisted Chevalley groups normality of the elementary subgroup was already proved in \cite{bak-vavilov, suzuki}. See also the survey \cite{survey} for details.

We are interested in the following more general situation. Suppose that \(G\) is a reductive group scheme over \(K\) such that it is sufficiently isotropic (in the sense of \cite{sga3} or \cite{petrov-stavrova}) Zariski locally, i.e. at all localizations at prime ideals. A typical example is the twisted general linear group
\(
    G(K) = \Aut_K(P)
\), where \(P\) is a finite projective \(K\)-module without direct summands and of the rank at least \(3\) at each point. Even the definition of the elementary subgroup for such groups is non-trivial since \(P\) does not have any unimodular vectors. We construct the elementary subgroup as an abstract group in theorem \ref{e-discr} and describe it in an elementary (though not entirely constructive) way. In theorem \ref{elem-gen} we show that the elementary subgroup may be generated by a scheme morphism (as in the case of Chevalley groups) and it is normal. The next theorem \ref{par-gen} states that our elementary groups coincide with the groups defined in \cite{petrov-stavrova} if the global isotropic rank is positive. Finally, the perfectness of the elementary group is proved in theorem \ref{perfect}.

Actually, our ultimate goal is to study Steinberg groups over locally isotropic reductive groups. Even in such generality centrality of \(\kfunctor_2\)-functor may be proved using a suitable localization technique \cite{linear-k2, thesis}, but we still have to construct the Steinberg group and \(\kfunctor_2\)-functor. This problem it technically more difficult than the case of \(\kfunctor_1\), so in this paper we work only with the elementary subgroup using the localization-and-patching method from \cite{yoga}. As in \cite{linear-k2, thesis} we use homotopes of the base ring instead of prime ideals and the formalism of ind- and pro-completion of categories.
Homotopes 
allow us to prove various functorial properties of the elementary subgroups, in particular we are able to use generic elements of various affine schemes. Ind- and pro-completions are necessary to adapt the localization-and-patching method to groups defined by homotopes instead of principal ideals, since such groups are not subgroups of the base group.

\section{Isotropic pinnings}

All rings and algebras in this paper are commutative and associative, but not necessarily unital. Throughout the paper we fix a unital ring \(K\). If \(\mathbf C\) is a category, then
\(
    X \in \mathbf C
\) means that \(X\) is an object of \(\mathbf C\). The categories of sets, groups, unital \(K\)-algebras, and non-unital \(K\)-algebras are denoted by \(\Set\), \(\Group\), \(\Ring_K\), and \(\Rng_K\) respectively. We identify schemes over \(K\) and corresponding covariant functors
\(
    \Ring_K \to \Set
\). A principal open subscheme
\(
    \Spec(K_s) \subseteq \Spec(K)
\) is also denoted by
\(
    \mathcal D(s)
\).

In this paper root systems are always crystallographic, but neither irreducible nor reduced in general, i.e. we admit components of types
\(
    \mathsf{BC}_\ell
\) for
\(
    \ell \geq 1
\). Let us say that roots \(\alpha\) and \(\beta\) in a root system are \textit{neighbors} if they are linearly independent, non-orthogonal, and there are no roots in the open angle
\(
    \mathbb R_{> 0} \alpha + \mathbb R_{> 0} \beta
\).

Let \(\Phi\) and \(\Psi\) be two root systems. A map
\(
    f \colon \Phi \cup \{0\} \to \Psi \cup \{0\}
\) is called a \textit{morphism} of root systems if it is induced by a linear map
\(
    \mathbb R \Phi \to \mathbb R \Psi
\) of the ambient spaces. A morphism \(f\) of root systems is called a \textit{factor-morphism} if it is surjective. The next lemma implies that for any morphism
\(
    f \colon \Phi \cup \{0\} \to \Psi \cup \{0\}
\) and any irreducible component
\(
    \Phi' \subseteq \Phi
\) not in the kernel of \(f\) there is unique irreducible component
\(
    \Psi' \subseteq \Psi
\) such that
\(
    f(\Phi') \subseteq \Psi' \cup \{0\}
\).

\begin{lemma} \label{root-sys-dec}
    Let \(\Phi\) be an irreducible root system. Then it cannot be covered by two hyperplanes of the ambient space (i.e. by subspaces of codimension \(1\)). If
    \(
        H \leq \mathbb R \Phi
    \) is a hyperplane and
    \(
        \alpha \in \Phi \cap H
    \), then \(\alpha\) has a neighbor in
    \(
        \Phi \setminus H
    \).
\end{lemma}
\begin{proof}
    We show that \(\Phi\) cannot be covered by a union of proper subspaces
    \(
        V_1 \cup V_2
    \) by induction on
    \(
        \dim(V_1 \cap V_2)
    \). Clearly, the subsets
    \(
        \Phi_1 = \Phi \cap V_1 \setminus V_2
    \) and
    \(
        \Phi_2 = \Phi \cap V_2 \setminus V_1
    \) are orthogonal, so the case
    \(
        \Phi \cap V_1 \cap V_2 = \varnothing
    \) (and the induction base
    \(
        \dim(V_1 \cap V_2) = 0
    \)) follows from the irreducibility. Take any
    \(
        \alpha \in \Phi \cap V_1 \cap V_2
    \). If \(\alpha\) is not orthogonal to \(\Phi_1\), then it lies in the span of \(\Phi_1\) and orthogonal to \(\Phi_2\). In other words,
    \(
        \Phi
        \subseteq V_1 \cup (V_2 \cap \alpha^\perp)
    \) or
    \(
        \Phi
        \subseteq (V_1 \cap \alpha^\perp) \cup V_2
    \), where
    \(
        \alpha^\perp
    \) is the orthogonal complement to \(\alpha\). Since
    \(
        \dim(V_1 \cap V_2 \cap \alpha^\perp)
        = \dim(V_1 \cap V_2) - 1
    \), we conclude by the induction assumption.

    To prove the second assertion, note that \(\Phi\) cannot be covered by \(H\) and
    \(
        \alpha^\perp
    \).
\end{proof}

Now let \(G\) be a reductive group scheme over \(K\). Recall \cite[Exp. XXIII, definition 1.1]{sga3} (or \cite[definition 6.1.1]{conrad}) that a \textit{pinning} of \(G\) consists of
\begin{itemize}

    \item a maximal torus
    \(
        T \leq G
    \);

    \item an isomorphism
    \(
        T \cong \mathbb G_\mult^k
    \) such that the roots and coroots are constant elements of
    \(
        \mathbb Z^k
    \) and the dual lattice
    \(
        \up k {\mathbb Z}
    \) respectively, the root subspaces of the Lie algebra \(\mathfrak g\) of \(G\) are trivial linear bundles;

    \item a basis
    \(
        \Delta \subseteq \Phi
    \);

    \item trivializing sections
    \(
        x_\alpha \in \mathfrak g_\alpha
    \) for
    \(
        \alpha \in \Delta
    \), where
    \(
        \mathfrak g_\alpha
    \) is the \(\alpha\)-th weight submodule of the Lie algebra \(\mathfrak g\).

\end{itemize}
We say that an \textit{isotropic pinning} of \(G\) is a tuple of
\begin{itemize}

    \item a torus
    \(
        T \leq G
    \);

    \item an isomorphism
    \(
        T \cong \mathbb G_\mult^k
    \) such that the roots are constant elements of
    \(
        \mathbb Z^k
    \) and form a root system \(\Psi\);

    \item a basis
    \(
        \Delta \subseteq \Psi
    \);

    \item isomorphisms
    \(
        \mathfrak g_\alpha \cong K^{m_\alpha}
    \) for all
    \(
        \alpha \in \Psi
    \)

\end{itemize}
with the following property: there is an fppf extension
\(
    K \subseteq K'
\) and a pinning of \(G_{K'}\) with the maximal torus
\(
    T' \cong \mathbb G_\mult^{k'}
\) and the root system \(\Phi\) such that
\(
    T_{K'} \leq T'
\) and this inclusion is given by a constant homomorphism of the weight lattices
\(
    \mathbb Z^{k'} \to \mathbb Z^k
\), so the induced map
\(
    \Phi \cup \{0\} \to \Psi \cup \{0\}
\) is surjective (i.e. a factor-morphism). The \textit{rank} of an isotropic pinning is smallest rank of the components of \(\Psi\) if no component of \(\Phi\) maps to zero and \(0\) otherwise. In other words, the rank is positive if and only if \(T\) does not commute with all simple factors of
\(
    G / \Cent(G)
\) at every geometric point.

An isotropic pinning is usually denoted just by
\(
    (T, \Psi)
\). By definition, every pinning may be considered as an isotropic pinning (under some choice of bases of
\(
    \mathfrak g_\alpha
\) for
\(
    \alpha \in \Phi \setminus \Delta
\)). An isotropic pinning
\(
    (T, \Psi)
\) is \textit{contained} in an isotropic pinning
\(
    (T', \Psi')
\) if
\(
    T \leq T'
\) and the inclusion is given by a constant homomorphism of the weight lattices
\(
    \mathbb Z^{k'} \to \mathbb Z^k
\), so the induced map
\(
    \Psi' \cup \{0\} \to \Psi \cup \{0\}
\) is a factor-morphism and the rank of
\(
    (T, \Psi)
\) is at most the rank of
\(
    (T', \Psi')
\). Each isotropic pinning is contained in a pinning after an fppf extension.

Recall from \cite[\S 2.1]{thesis} and \cite[\S 4]{twisted-forms} that a \textit{\(2\)-step nilpotent \(K\)-module}
\(
    (M, M_0)
\) consists of
\begin{itemize}

    \item a group \(M\) with the group operation \(\dotplus\) and a \(2\)-step nilpotent filtration
    \(
        M_0 \leq M
    \) (i.e.
    \(
        [M, M] \leq M_0
    \) and
    \(
        [M, M_0] = \dot 0
    \));

    \item a left \(K\)-module structure on \(M_0\);

    \item a right action
    \(
        ({-}) \cdot ({=}) \colon M \times K \to M
    \) of the multiplicative monoid \(K^\bullet\) by group endomorphisms

\end{itemize}
such that
\begin{itemize}

    \item
    \(
        [m \cdot k, m' \cdot k'] = kk' [m, m']
    \);

    \item
    \(
        m \cdot (k + k')
        = m \cdot k
        \dotplus kk' \tau(m)
        \dotplus m \cdot k'
    \) for some (uniquely determined)
    \(
        \tau(m) \in M_0
    \);

    \item
    \(
        m_0 \cdot k = k^2 m_0
    \) for
    \(
        m_0 \in M_0
    \).

\end{itemize}
It easily follows that
\(
    m \cdot 0 = \dot 0
\),
\(
    m \cdot (-k) = k^2 \tau(m) \dotminus m \cdot k
\),
\(
    \tau(\dot 0) = \dot 0
\),
\(
    \tau(m \dotplus m')
    = \tau(m) \dotplus [m, m'] \dotplus \tau(m')
\),
\(
    \tau(m \cdot k) = k^2 \tau(m)
\),
\(
    \tau(\dotminus m) = -\tau(m)
\), and
\(
    \tau(m_0) = 2m_0
\) for
\(
    m_0 \in M_0
\). Every \(K\)-module \(M\) admits two natural structures of a \(2\)-step nilpotent \(K\)-module with the smallest nilpotent filtration
\(
    M_0 = 0
\) and the largest one
\(
    M_0 = M
\). Scalar extensions of \(2\)-step nilpotent modules are defined in \cite[\S 2.1]{thesis} and \cite[\S 4]{twisted-forms}.

Let
\(
    (M, M_0)
\) and
\(
    (N, N_0)
\) be \(2\)-step nilpotent \(K\)-modules. A map
\(
    q \colon M \to N
\) is called \textit{\(K\)-quadratic} if
\begin{itemize}

    \item
    \(
        q(M_0) \subseteq N_0
    \) and
    \(
        q|_{M_0} \colon M_0 \to N_0
    \) is \(K\)-linear,

    \item
    \(
        q(m \dotplus m') = q(m) + b(m, m') + q(m')
    \) for some (uniquely determined) \(K\)-bilinear map
    \(
        b \colon M / M_0 \times M / M_0 \to N_0
    \),

    \item
    \(
        q(m \cdot k) = q(m) \cdot k
    \).

\end{itemize}
It follows that
\(
    q(\dot 0) = \dot 0
\),
\(
    q(\dotminus m) = b(m, m) \dotminus q(m)
\),
\(
    q(\tau(m)) = \tau(q(m)) \dotminus b(m, m)
\), and
\(
    q([m, m'])
    = [q(m), q(m')]
    \dotplus b(m, m')
    \dotminus b(m', m)
\).

Let
\(
    (T, \Psi)
\) be an isotropic pinning of \(G\). We claim that there exist (necessarily unique) \textit{root subgroups}
\(
    U_\alpha \leq G
\) for \(
    \alpha \in \Psi
\) satisfying the following properties:
\begin{itemize}

    \item The subgroups \(U_\alpha\) are preserved under base change, they are smooth closed subschemes.

    \item If
    \(
        (T, \Psi)
    \) is a pinning, then \(U_\alpha\) are the ordinary root subgroups.

    \item If
    \(
        (T, \Psi) \subseteq (T', \Psi')
    \), then
    \(
        U_\alpha
        = \prod_{\substack{
            \alpha' \in \Psi'\\
            \Image(\alpha') \in \{\alpha, 2 \alpha\}
        } } U_{\alpha'}
    \) for every
    \(
        \alpha \in \Psi
    \).

\end{itemize}

If \(U_\alpha\) exists, then by descent the action of \(T\) on \(U_\alpha\) factors through
\(
    \alpha \colon T \to \mathbb G_\mult
\) and also through
\(
    \mathbb G_\mult \to \mathbb A^1
\) (in the unique way), where
\(
    \mathbb A^1
\) is considered as a multiplicative monoid. We construct the root subgroups as follows:
\begin{itemize}

    \item If
    \(
        \alpha \in \Psi \setminus 2 \Psi
    \), then
    \(
        U_\alpha = \bigcap_\lambda U_G(\lambda)
    \), where \(\lambda\) runs over all constant coweights of \(T\) such that
    \(
        \langle \alpha, \lambda \rangle > 0
    \). Here the subfunctor
    \(
        U_G(\lambda)(K')
        = \{
            g \in G(K')
        \mid
            \lim_{t \to 0} \up{\lambda(t)}{g} = 1
        \}
    \) is a closed subscheme of \(G\), see \cite[theorem 4.1.7]{conrad} for details.

    \item If
    \(
        \alpha \in \Psi \cap 2 \Psi
    \), then
    \(
        U_\alpha \leq U_{\alpha / 2}
    \) is the locus where the action of
    \(
        \mathbb A^1
    \) on
    \(
        U_{\alpha / 2}
    \) factors through the squaring map
    \(
        \mathbb A^1 \to \mathbb A^1
    \). In other words, for every \(K'\) the set
    \(
        U_\alpha(K')
    \) consists of
    \(
        g \in U_{\alpha / 2}(K')
    \) such that the morphism
    \(
        \mathbb A^1_{K'} \to G_{K'},\,
        t \mapsto \up tg
    \) corresponding to
    \(
        \alpha / 2
    \) factors through
    \(
        \mathbb A^1_{K'} \to \mathbb A^1_{K'}
    \).

\end{itemize}

Moreover, we have the following:
\begin{itemize}

    \item If
    \(
        \alpha \in \Psi \setminus \frac 1 2 \Psi
    \), then \(U_\alpha\) is abelian with the Lie algebra
    \(
        \mathfrak g_\alpha
    \) and there is a unique \(T\)-equivariant isomorphism
    \(
        t_\alpha
        \colon \mathbb G_\add^{m_\alpha}
        \to U_\alpha
    \) inducing the coordinatization isomorphism on the Lie algebras.

    \item If
    \(
        \alpha \in \Psi \cap \frac 1 2 \Psi
    \), then \(U_\alpha\) is \(2\)-step nilpotent with the nilpotent filtration
    \(
        U_{2 \alpha} \leq U_\alpha
    \) and the Lie algebra
    \(
        \mathfrak g_\alpha
        \oplus \mathfrak g_{2 \alpha}
    \),
    \(
        U_\alpha / U_{2 \alpha}
    \) is representable, and there is a unique \(T\)-equivariant isomorphism
    \(
        t_\alpha
        \colon \mathbb G_\add^{m_\alpha}
        \to U_\alpha / U_{2 \alpha}
    \) inducing the coordinatization isomorphism on the Lie algebras.

    \item Consider the fppf sheaf \(U_\alpha\) on the category of unital \(K\)-algebras if
    \(
        \alpha \in \Psi \cap \frac 1 2 \Psi
    \). By descent, it is a sheaf of \(2\)-step nilpotent modules. Using \cite[lemma 13 and proposition 3]{thesis} (or \cite[proposition 2]{twisted-forms}) it is easy to see that the \(2\)-step nilpotent \(K\)-module
    \(
        U_\alpha(K)
    \) is \textit{universal}, i.e. its scalar extensions are always defined, and such scalar extensions are the values of \(U_\alpha\). Since
    \(
        U_\alpha(K) / U_{2 \alpha}(K)
    \) is a free \(K\)-module,
    \(
        U_\alpha(K)
    \) splits, i.e. there is an isomorphism
    \(
        t_\alpha
        \colon \mathbb G_\add^{m_{2 \alpha}}
        \dotoplus \mathbb G_\add^{m_\alpha}
        \to U_\alpha
    \) of sheaves of \(2\)-step nilpotent modules (such a structure on
    \(
        \mathbb G_\add^{m_{2 \alpha}}
        \dotoplus \mathbb G_\add^{m_\alpha}
    \) is given by some \(K\)-bilinear \(2\)-cocycle
    \(
        c
        \colon K^{m_\alpha} \times K^{m_\alpha}
        \to K^{m_{2 \alpha}}
    \)). Of course,
    \(
        t_\alpha|_{\mathbb G_\add^{m_{2 \alpha}}}
        = t_{2 \alpha}
    \).

    \item The scheme centralizer \(L\) of \(T\) in \(G\) is a reductive closed subgroup of \(G\) with the Lie algebra
    \(
        \mathfrak g_0
    \). In particular, it is smooth with connected fibers. Moreover, \(L\) normalizes all \(U_\alpha\).

    \item Let
    \(
        \Pi \subseteq \Psi
    \) be the subset of positive roots. The product map
    \(
        \prod_{
            \alpha \in \Pi \setminus 2 \Pi
        } U_{-\alpha}
        \times L
        \times \prod_{
            \alpha \in \Pi \setminus 2 \Pi
        } U_\alpha
        \to G
    \) is an open embedding. The group schemes
    \(
        U^{\pm} = \prod_{
            \alpha \in \Pi \setminus 2 \Pi
        } U_{\pm \alpha}
    \) are the unipotent radicals of opposite parabolic subgroups
    \(
        P^{\pm} = L U^{\pm}
    \) and \(L\) is the common Levi subgroup of \(P^{\pm}\). Moreover,
    \(
        P^{\pm} = P_G(\pm \lambda)
    \) for any constant coweight \(\lambda\) such that
    \(
        \langle \alpha, \lambda \rangle > 0
    \) for all
    \(
        \alpha \in \Delta
    \), see \cite[theorem 4.1.7 and example 5.2.2]{conrad}.

\end{itemize}

We denote the domain of \(t_\alpha\) by \(P_\alpha\), i.e.
\(
    P_\alpha = \mathbb G_\add^{m_\alpha}
\) for
\(
    \alpha \in \Psi \setminus \frac 1 2 \Psi
\) and
\(
    P_\alpha
    = \mathbb G_\add^{m_{2 \alpha}}
    \dotoplus \mathbb G_\add^{m_\alpha}
\) for
\(
    \alpha \in \Psi \cap \frac 1 2 \Psi
\).

The Chevalley commutator formula is
\[
    [t_\alpha(x), t_\beta(y)]
    = \prod_{\substack{
        i \alpha + j \beta \in \Psi \\ i, j > 0
    } } t_{i \alpha + j \beta}(
        f_{\alpha \beta i j}(x, y)
    )
\]
for some uniquely determined scheme morphisms
\(
    f_{\alpha \beta i j}
\) if \(\alpha\) and \(\beta\) are linearly independent. Here we fix some order in the product and assume that
\(
    f_{\alpha, \beta, 2i, 2j} = 0
\) if both
\(
    i \alpha + j \beta
\) and
\(
    2i \alpha + 2j \beta
\) lie in \(\Psi\). The morphisms
\(
    f_{\alpha \beta i j}
\) are \(L\)-equivariant. Moreover,
\begin{itemize}

    \item If all \(\alpha\), \(\beta\),
    \(
        i \alpha + j \beta
    \) are not ultrashort (i.e. roots of a component of type
    \(
        \mathsf{BC}_\ell
    \) with the smallest length), then the expression
    \(
        f_{\alpha \beta i j}(x, y)
    \) is polynomial, homogeneous on \(x\) of degree \(i\) and on \(y\) of degree \(j\).

    \item If \(\alpha\) and \(\beta\) are ultrashort in a common irreducible component of \(\Psi\), then
    \(
        f_{\alpha \beta 1 1}
    \) factors through a bilinear morphism
    \(
        P_\alpha / P_{2 \alpha}
        \times P_\beta / P_{2 \beta}
        \to P_{\alpha + \beta}
    \).

    \item Suppose that \(\alpha\) is short, \(\beta\) is ultrashort, and
    \(
        (\alpha, \beta)
    \) is a base of a root subsystem of type
    \(
        \mathsf{BC}_2
    \). Then
    \(
        f_{\alpha \beta 1 2}(x, y)
    \) is \(K\)-linear on \(x\) and \(K\)-quadratic on \(y\) (with respect to the largest nilpotent filtration on
    \(
        P_{\alpha + 2 \beta}
    \)),
    \(
        f_{\alpha \beta 1 1}(x, y)
    \) is \(K\)-quadratic on \(x\) (with respect to the smallest nilpotent filtration on \(P_\beta\)) and a homomorphism on \(y\).

\end{itemize}

The last claim easily follows from the identities
\(
    [xy, z] = \up{x}{[y, z]}\, [x, z]
\) and
\(
    [x, yz] = [x, y]\, \up{y}{[x, z]}
\). For every root \(\alpha\) the canonical morphism
\(
    P_\alpha \times \mathbb A^1 \to P_\alpha
\) is denoted by
\(
    ({-}) \cdot ({=})
\) (it is the usual module structure in the case
\(
    \alpha \notin \frac 1 2 \Psi
\)), so always
\(
    f_{\alpha \beta i j}(x \cdot k, y \cdot k')
    = f_{\alpha \beta i j}(x, y) \cdot k^i {k'}^j
\).

\begin{lemma} \label{loc-iso-pin}
    If \(K\) is semilocal with connected spectrum, then every parabolic subgroup contains a minimal one and every isotropic pinning is contained in a maximal one. If
    \(
        (T, \Psi)
    \) and
    \(
        (T', \Psi')
    \) are two maximal isotropic pinnings, then there is
    \(
        g
        \in U_{T, \Psi}^+(K)\,
        U_{T, \Psi}^-(K)\,
        U_{T, \Psi}^+(K)
    \) such that
    \(
        T' = \up gT
    \) and
    \(
        \Psi' = \up g{\Psi}
    \). Here
    \(
        U^{\pm}_{T, \Psi}
    \) are the unipotent radicals of the canonical parabolic subgroups associated with
    \(
        (T, \Psi)
    \). Every minimal parabolic subgroup may be constructed by a maximal isotropic pinning.
\end{lemma}
\begin{proof}
    See \cite[Exp. XXVI, corollary 5.2, corollary 5.7, proposition 6.16, theorem 7.13]{sga3}.
\end{proof}

If \(K\) is semilocal with connected spectrum, then the \textit{isotropic rank} of \(G\) is the common rank of its maximal isotropic pinnings. The \textit{local isotropic rank} of \(G\) over any unital ring \(K\) is the minimum of the isotropic ranks of
\(
    G_{\mathfrak m}
\) for all maximal ideals
\(
    \mathfrak m \leqt K
\).

\section{Ind-pro-completion}

Let \(\mathbf C\) be a category. Its \textit{ind-completion}
\(
    \Ind(\mathbf C)
\) is the universal category with a functor from \(\mathbf C\) containing all direct limits indexed by small filtered categories (e.g. by directed sets, especially \(\mathbb N\)). We use the following standard construction of
\(
    \Ind(\mathbf C)
\): its objects are \textit{direct systems} in \(\mathbf C\), i.e. covariant functors
\(
    X \colon \mathbf I_X \to \mathbf C
\) for small filtered \textit{index categories}
\(
    \mathbf I_X
\). Morphisms are given by the formula
\[
    \Ind(\mathbf C)(X, Y)
    = \varprojlim\nolimits_{i \in \mathbf I_X}^\Set
        \varinjlim\nolimits_{j \in \mathbf I_Y}^\Set
            \mathbf C(X(i), Y(j)).
\]
For the definition of composition see \cite[definition 6.1.1]{kashiwara-schapira}.

Dually, the \textit{pro-completion} of \(\mathbf C\) is
\(
    \Pro(\mathbf C) = \Ind(\mathbf C^\op)^\op
\), it is the universal category with a functor from \(\mathbf C\) containing all projective limits indexed by small filtered categories. Objects of
\(
    \Pro(\mathbf C)
\) are \textit{inverse systems} in \(\mathbf C\), i.e. contravatiant functors
\(
    X \colon \mathbf I_X \to \mathbf C
\) for small filtered index categories \(\mathbf I_X\). Morphisms are given by
\[
    \Pro(\mathbf C)(X, Y)
    = \varprojlim\nolimits_{j \in \mathbf I_Y}^\Set
        \varinjlim\nolimits_{i \in \mathbf I_X}^\Set
            \mathbf C(X(i), Y(j)).
\]

If
\(
    X \colon \mathbf I_X \to \mathbf C
\) is a direct system and
\(
    u \colon \mathbf I'_X \to \mathbf I_X
\) is a cofinal, then the induced morphism
\(
    X \circ u \to X
\) is an isomorphism in
\(
    \Ind(\mathbf C)
\). Dually, if \(X\) is an inverse system, then the induced morphism
\(
    X \to X \circ u
\) is an isomorphism in
\(
    \Pro(\mathbf C)
\). A \textit{level morphism} in
\(
    \Ind(\mathbf C)
\) or
\(
    \Pro(\mathbf C)
\) is an object of the ind-completion or the pro-completion of the category of arrows in \(\mathbf C\). Every morphism in
\(
    \Ind(\mathbf C)
\) or
\(
    \Pro(\mathbf C)
\) is isomorphic to a level one in the category of arrows \cite[proposition 6.1.14]{kashiwara-schapira}.

There are canonical fully faithful functors
\(
    \mathbf C \to \Ind(\mathbf C)
\) and
\(
    \mathbf C \to \Pro(\mathbf C)
\), they map
\(
    X \in \mathbf C
\) to the corresponding functor indexed by the terminal category with unique morphism. If \(\mathbf C\) is finitely complete or cocomplete, then
\(
    \Ind(\mathbf C)
\) and
\(
    \Pro(\mathbf C)
\) are also finitely complete or cocomplete respectively by \cite[corollary 6.1.17(i) and proposition 6.1.18(iii)]{kashiwara-schapira}. If \(\mathbf C\) is finitely complete, then in
\(
    \Ind(\mathbf C)
\) direct limits commute with finite limits. A limit and a colimit of a finite level diagram may be calculated levelwise.

If
\(
    F \colon \mathbf C \to \mathbf D
\) is a functor, then it continues to unique (up to unique natural isomorphism) functors
\(
    \Ind(F) \colon \Ind(\mathbf C) \to \Ind(\mathbf D)
\) and
\(
    \Pro(F) \colon \Pro(\mathbf C) \to \Pro(\mathbf D)
\) commuting with direct limits and projective limits respectively \cite[proposition 6.1.9, 6.1.10]{kashiwara-schapira}. If \(F\) is fully faithful, then
\(
    \Ind(F)
\) and
\(
    \Pro(F)
\) are also fully faithful.

Recall that a \textit{regular epimorphism} in a category \(\mathbf C\) is a coequalizer of a couple of morphisms. A \textit{kernel pair} of a morphism
\(
    f \colon X \to Y
\) is the limit
\(
    \lim(X \xrightarrow f Y \xleftarrow f X)
    \rightrightarrows X
\). A category \(\mathbf C\) is called \textit{regular} if it is finitely complete, all kernel pairs have coequalizers, and the class of regular epimorphisms is closed under base change. If \(\mathbf C\) is regular, then any morphism
\(
    f \colon X \to Y
\) has unique (up to unique isomorphism) \textit{image decomposition}
\(
    X \xrightarrow u \Image(f) \xrightarrow v Y
\), where \(u\) is a regular epimorphism and \(v\) is a monomorphism. We usually denote the image
\(
    \Image(f)
\) just by \(f(X)\). The image decomposition is functorial on \(f\).

The pro-completion and the ind-completion of a regular category are regular by \cite[example 1.11 and \S 1.8]{jacqmin-janelidze}. Moreover, image decompositions of level morphisms may be computed levelwise. In particular, if the components of a level morphism \(f\) are monomorphisms or regular epimorphisms, then \(f\) is itself a monomorphism or a regular epimorphism respectively. A level morphism \(f \in \Pro(\mathbf C)(X, Y)\) in the pro-completion of a regular category is a regular epimorphism if and only if for all \(i \in \mathbf I_X\) there are \(j \in \mathbf I_X\) and \(\varphi \in \mathbf I_X(i, j)\) such that \(\Image(Y(\varphi)) \subseteq \Image(f(i))\) as subobjects of \(Y(i)\). This easily follows from \cite[proposition 1.7]{dydak-portal} applied to \(\Image(f) \subseteq Y\).

A regular category \(\mathbf C\) is called \textit{coherent} if the families of subobjects
\(
    \Sub(X)
\) are bounded \(\vee\)-semilattices for all
\(
    X \in \mathbf C
\) and for all
\(
    f \in \mathbf C(X, Y)
\) the base change maps
\(
    f^* \colon \Sub(Y) \to \Sub(X)
\) are homomorphisms of bounded \(\vee\)-semilattices. The category \(\Set\) is coherent. We denote the least upper bound of subobjects
\(
    A, B \subseteq X
\) by
\(
    A \cup B
\) and the smallest subobject by
\(
    \varnothing \subseteq X
\). By \cite[the list after example 1.12 and \S 1.8]{jacqmin-janelidze} if \(\mathbf C\) is coherent and finitely cocomplete, then
\(
    \Ind(\mathbf C)
\) and
\(
    \Pro(\mathbf C)
\) also have this property. 

It follows that
\(
    \Ind(\Pro(\Set))
\) is a finitely cocomplete coherent category. 
The categories \(\Set\),
\(
    \Ind(\Set)
\),
\(
    \Pro(\Set)
\) are its full subcategories and the embedding functors preserve finite limits, finite colimits, monomorphisms, regular epimorphisms, the image decomposition, and finite unions of subobjects.

In any finitely complete category \(\mathbf C\) we often write compositions of morphisms by first-order terms, e.g. the identities
\(
    (xy)z = x(yz)
\),
\(
    x 1 = x = 1 x
\),
\(
    x x^{-1} = 1 = x^{-1} x
\) for
\(
    ({-}) ({=}) \colon X \times X \to X
\),
\(
    ({-})^{-1} \colon X \to X
\),
\(
    1 \colon \{*\} \to X
\) mean that \(X\) is a group object (here \(\{*\}\) denotes the terminal object in \(\mathbf C\)). We write
\(
    x \in X
\) if a formal variable \(x\) has the domain
\(
    X \in \mathbf C
\). An \textit{action} of a group object \(G\) on a group object \(H\) is a morphism
\(
    \up{({-})}{({=})} \colon G \times H \to H
\) such that
\(
    \up{gg'}{h} = \up g{(\up{g'}h)}
\),
\(
    \up 1h = h
\), and
\(
    \up{g}{(hh')} = \up g h\, \up{g}{h'}
\). In this case the \textit{semidirect product}
\(
    H \rtimes G
\) is their usual product
\(
    H \times G
\) in \(\mathbf C\) with the group object structure given by
\(
    (h \rtimes g) (h' \rtimes g')
    = h\, \up{g}{h'} \rtimes gg'
\). A \textit{crossed module} is a homomorphism
\(
    \delta \colon X \to G
\) of group objects together with an action of \(G\) on \(X\) such that \(\delta\) is \(G\)-equivariant and
\(
    \up xy = \up{\delta(x)}y
\) for
\(
    x, y \in X
\). By the Yoneda lemma elementary properties of such algebraic objects follow from the abstract group theory.

\begin{lemma} \label{subgr-gen}
    Let \(\mathbf C\) be a finitely cocomplete coherent category, \(G\) be a group object in
    \(
        \Ind(\mathbf C)
    \), and
    \(
        X \subseteq G
    \) be a subobject of \(G\). Then
    \(
        \langle X \rangle
        = \varinjlim_{n \in \mathbb N}
            (X \cup \{1\} \cup X^{-1})^n
    \) is the smallest group subobject of \(G\) containing \(X\), where \(X^n\) is the image of the multiplication morphism
    \(
        X^{\times n} \to G
    \), \(X^{-1}\) is the image of \(X\) under the inversion, and \(\{1\}\) is the image of the identity
    \(
        1 \colon \{*\} \to G
    \).
\end{lemma}
\begin{proof}
    Let
    \(
        H \leq G
    \) be a group subobject containing \(X\). Clearly,
    \(
        (X \cup \{1\} \cup X^{-1})^n \subseteq H
    \) for any
    \(
        n \in \mathbb N
    \), so
    \(
        \langle X \rangle \leq H
    \). On the other hand, the group operations of \(G\) induce group operations on
    \(
        \langle X \rangle
    \) since direct limits commute with finite products in
    \(
        \Ind(\mathbf C)
    \).
\end{proof}

Non-unital ring objects in a finitely complete category are defined in the same way. If \(A\) and \(B\) are ring objects in \(\mathbf C\), then an \textit{action} of \(A\) on \(B\) is given by a biadditive morphism
\(
    ({-}) ({=}) \colon A \times B \to B
\) such that
\(
    a(bb') = (ab)b'
\) and
\(
    (aa')b = a(a'b)
\). A \textit{unital action} of a unital \(A\) on \(B\) satisfies also the additional axiom
\(
    1 b = b
\). The term \textit{\(A\)-algebra} always means a ring object with a unital action of \(A\). If \(A\) and \(B\) are \(R\)-algebras for some unital ring object \(R\), then we also require that
\(
    (ra) b = r(ab) = a(rb)
\) for
\(
    r \in R
\),
\(
    a \in A
\),
\(
    b \in B
\) for an action of \(A\) on \(B\).

If \(A\) acts on \(B\), then
\(
    A \rtimes B
\) is their \textit{semi-direct product}, it is their product as abelian group objects together with the multiplication
\(
    (a \rtimes b) (a' \rtimes b')
    = (aa' + ab' + a'b) \rtimes bb'
\). A \textit{crossed module} of ring objects is a homomorphism
\(
    \delta \colon X \to A
\) together with an action of \(A\) on \(X\) such that
\(
    \delta(ax) = a \delta(x)
\) and
\(
    xy = \delta(x) y
\) for
\(
    x, y \in X
\) and
\(
    a \in A
\).

\section{Ind-pro-algebras}

Recall that \(K\) denotes the base unital ring. For any
\(
    s, t \in K
\) and
\(
    A \in \Rng_K
\) consider the
\(
    (A \rtimes K)
\)-module
\[
    \tfrac 1 s A^{(t)}
    = \{\tfrac{a^{(t)}}s \mid a \in A\}
\]
with the operations
\(
    \frac{a^{(t)}}s + \frac{b^{(t)}}s
    = \frac{(a + b)^{(t)}}s
\) and \(
    \frac{a^{(t)}}s (b \rtimes k)
    = \frac{(ab + ak)^{(t)}}s
\). There are natural homomorphisms
\[
    \tfrac 1 s A^{(tt')}
    \to \tfrac 1{ss'} A^{(t)},\,
    \tfrac{a^{(tt')}}s
    \mapsto \tfrac{(at's')^{(t)}}{ss'}
\]
and bilinear multiplication maps
\[
    ({-}) ({=})
    \colon \tfrac 1s A^{(t)}
    \times \tfrac 1{s'} A^{(t')}
    \to \tfrac 1{ss'} A^{(tt')},\,
    (\tfrac{a^{(t)}}s, \tfrac{b^{(t')}}{s'})
    \mapsto \tfrac{(ab)^{(tt')}}{ss'}.
\]
We omit the upper index \((1)\) and the denominator \(1\) since
\(
    \frac 1 1 A^{(1)} \cong a,\,
    \frac{a^{(1)}}{1} \mapsto a
\) is an algebra isomorphism. Let the \textit{formal localization} of \(A\) at \(s\) be
\[
    A_s^{\Ind} = \varinjlim\nolimits^{\Ind(\Set)} (
        A
        \to \tfrac 1 s A
        \to \tfrac{1}{s^2} A
        \to \ldots
    ),
\]
so \(A_s\) is the ordinary direct limit of the direct system
\(
    A_s^{\Ind}
\). The object
\(
    A_s^{\Ind}
\) is a non-unital algebra over
\(
    A \rtimes K
\) in
\(
    \Ind(\Set)
\) (it is unital if \(A\) is unital),
\(
    A_s^{\Ind} \to A_{ss'}^{\Ind}
\) are homomorphisms of algebras. Similarly, the \textit{colocalization} of \(A\) at \(s\) is
\[
    A^{(s^\infty)}
    = \varprojlim\nolimits^{\Pro(\Set)} (
        \ldots \to A^{(s^2)} \to A^{(s)} \to A
    ),
\]
it is a non-unital algebra over
\(
    A \rtimes K
\) in
\(
    \Pro(\Set)
\),
\(
    A^{((ss')^\infty)} \to A^{(s^\infty)}
\) are algebra homomorphisms. Actually,
\(
    A^{(s^\infty)}
\) is a pro-\(
    (A \rtimes K)
\)-algebra since
\(
    A^{(s)}
\) are algebras with respect to
\(
    a^{(s)} b^{(s)} = (abs)^{(s)}
\) called \textit{homotopes} of \(A\).

Now we construct the categories used in the construction of elementary groups. For any small category \(\mathbf D\) there is the category
\(
    \Cat(\mathbf D, \Set)
\) of functors
\(
    \mathbf D \to \Set
\) and natural transformations between them, it is also a finitely cocomplete coherent category (limits, colimits, the image decomposition, and finite unions of subobjects are calculated componentwise). We usually take
\(
    \mathbf D = \{*\}
\) in order to study an individual elementary group or
\(
    \mathbf D = \Ring^\fp_K
\) (the category of finitely presented unital \(K\)-algebras) to study the whole elementary group functor. Let
\[
    \IP_{\mathbf D}
    = \Ind(\Pro(\Cat(\mathbf D, \Set))).
\]
This construction is a contravariant functor on \(\mathbf D\), so for any object
\(
    D \in \mathbf D
\) there is a functor
\(
    \IP_{\mathbf D} \to \mathbf{IP}_{\{*\}}
\) of \textit{restriction to \(D\)}. On the other side, if \(\mathbf D\) is non-emtpy then there is a full subcategory of \textit{\(\mathbf D\)-constant objects}
\(
    \IP_{\{*\}} \subseteq \IP_{\mathbf D}
\) induced by
\(
    \mathbf D \to \{*\}
\). Let
\(
    \IPRng_{\mathbf D, s}
\) be the category of non-unital algebras in
\(
    \IP_{\mathbf D}
\) over the \(\mathbf D\)-constant unital ring object
\(
    K_s^{\Ind}
\), so
\(
    \IPRng_{\{*\}, s}
\) is the usual category of non-unital algebras over
\(
    K_s^{\Ind}
\) in
\(
    \Ind(\Pro(\Set))
\). Note that the set of global elements of
\(
    K_s^{\Ind}
\) (i.e.
\(
    \IP_{\mathbf D}(\{*\}, K_s^{\Ind})
\)) is canonically isomorphic to \(K_s\) if \(\mathbf D\) is connected.

Actually, we would like to take
\(
    \mathbf D = \Ring_K
\) to consider elementary groups for all unital \(K\)-algebras simultaneously. Unfortunately, the ``category''
\(
    \Cat(\Ring_K, \Set)
\) is too large (its objects are proper classes, not sets). Instead we may take the category of functors
\(
    \Ring_K \to \Set
\) preserving direct limits and all natural transformations, it is equivalent to
\(
    \Cat(\Ring_K^\fp, \Set)
\) by \cite[corollary 6.3.2]{kashiwara-schapira} since \(\Ring_K\) is an ind-completion of
\(
    \Ring_K^\fp
\) \cite[corollary 6.3.5]{kashiwara-schapira}. Hence for any
\(
    R \in \Ring_K
\) there is a functor
\(
    \IP_{\Ring^\fp_K} \to \IP_{\{*\}}
\) of \textit{restriction to \(R\)}. We may also work only with small categories by considering sets with bounded rank (sufficiently large for our further constructions) instead of the whole \(\Set\) and by taking only countable ind- and pro-completions.

For example, let \(\mathcal R\) be the inclusion functor
\(
    \Ring_K \to \Set
\) and
\(
    s \in K
\). Applying the colocalization construction to every component of \(\mathcal R\) (i.e. for every unital \(K\)-algebra) we obtain an object
\(
    \mathcal R^{(s^\infty)}
    \in \IPRng_{\Ring^\fp_K, s}
\). The action of
\(
    K^{\Ind}_s
\) on
\(
    \mathcal R^{(s^\infty)}
\) is given by the maps
\(
    ({-}) ({=})
    \colon \mathcal R^{(s^{n + m})}
    \times \frac 1{s^n} K
    \to \mathcal R^{(s^m)},\,
    (r^{(s^{n + m})}, \frac 1{s^n} k)
    \mapsto (rk)^{(s^m)}
\). We may consider this object as a functor
\(
    \Ring_K \to \IPRng_{\{*\}, s}
\), where the pro-structure (and the action of
\(
    K^{\Ind}_s
\)) are ``uniform'' on the algebra \(R\). The value of this functor on \(R\) (i.e. the restriction of
\(
    \mathcal R^{(s^\infty)}
\) to \(R\)) is precisely
\(
    R^{(s^\infty)}
\). It is easy to see that if
\(
    \mathcal D(s) = \mathcal D(s')
\) (i.e. \(s\) and \(s'\) divide some powers of each other), then
\(
    \mathcal R^{(s^\infty)}
    \cong \mathcal R^{({s'}^\infty)}
\) as ring objects in
\(
    \IP_{\Ring^\fp_K}
\). Instead of \(\mathcal R\) we may take any functor
\(
    \Ring_K \to \Rng_K
\) commuting with direct limits such as
\(
    \mathcal R^{(t)}
\) or
\(
    t\mathcal R
\).

Note that the composition
\(
    \mathcal R^{(s^\infty)}
    \to \mathcal R
    \to \mathcal R_s
\) is
\(
    K_s^{\Ind}
\)-equivariant. The following lemma shows that in the Noetherian case we may consider ideals instead of homotopes and such composition is a monomorphism for every individual algebra \(R\). But these properties are not ``uniform'' on the algebra, i.e. do not hold for the whole \(\mathcal R\).
\begin{lemma} \label{noeth-coloc}
    Suppose that \(K\) is Noetherian,
    \(
        A \in \Rng_K^\fp
    \) (i.e.
    \(
        A \rtimes K
    \) is a finitely generated unital \(K\)-algebra), and
    \(
        s \in K
    \). Then the natural level morphism
    \[
        A^{(s^\infty)}
        \to \varprojlim\nolimits^{\Pro(\Set)}_n
            s^n A,\,
        a^{(s^n)} \mapsto s^n a
    \]
    is an isomorphism of pro-sets. Moreover, the morphism
    \[
        \varprojlim\nolimits^{\Pro(\Set)}_n s^n A
        \to A_s,\,
        s^n a \mapsto \textstyle \frac{s^n a}1
    \]
    is a monomorphism of pro-sets.
\end{lemma}
\begin{proof}
    Replacing \(A\) by
    \(
        A \rtimes K
    \) we may assume that
    \(
        A = K
    \). Since the maps
    \(
        K^{(s^n)} \to s^n K
    \) are surjective, it suffices to check that
    \[
        K^{(s^\infty)} \to K_s,\,
        x^{(s^n)} \mapsto \textstyle \frac{s^n x}1
    \]
    is a monomorphism of pro-sets. Let
    \(
        I_n = \{x \in K \mid s^n k = 0\}
    \), so
    \(
        0 = I_0 \leqt I_1 \leqt \ldots \leqt K
    \) is an ascending chain of ideals. Since \(K\) is Noetherian it stabilizes at some index \(n_*\). It follows that
    \(
        \Ker(K^{(s^{n + n_*})} \to K_s)
        = I_{n_*}^{(s^{n + n_*})}
    \) maps to zero under
    \(
        K^{(s^{n + n_*})} \to K^{(s^n)}
    \) for all \(n\), so we may apply \cite[proposition 2.3]{dydak-portal}.
\end{proof}

Pro-rings
\(
    R^{(s^\infty)}
\) and the ring object
\(
    \mathcal R^{(s^\infty)}
\) are not unital. Fortunately, they satisfy a weaker condition sufficient for our purposes. A ring object \(R\) in an ind-completion of a finitely cocomplete coherent category (e.g. in
\(
    \IP_{\mathbf D}
\) or in \(\Set\)) is called \textit{idempotent} if
\(
    R = \langle xy \mid x, y \in R \rangle
\), the right hand sides denotes the additive subgroup generated by a subobject. An action of such \(R\) on a ring object \(A\) is \textit{unital} if
\(
    A = \langle xa \mid x \in R, a \in A \rangle
\). Finally, we say that \(R\) is \textit{power idempotent} if
\(
    R = \langle xy^k \mid x, y \in R \rangle
\) for every
\(
    k > 0
\).

\begin{lemma} \label{power-idem}
    For any
    \(
        s \in K
    \) the ring object
    \(
        \mathcal R^{(s^\infty)}
    \) is power idempotent. Moreover, for any functor
    \(
        \mathcal A \colon \Ring_K \to \Rng_K
    \) commuting with direct limits and with a structure of an \(\mathcal R\)-algebra (e.g.
    \(
        \mathcal A = t \mathcal R
    \) or
    \(
        \mathcal A = \mathcal R^{(t)}
    \) for
    \(
        t \in K
    \)) the action of
    \(
        \mathcal R^{(s^\infty)}
    \) on
    \(
        \mathcal A^{(s^\infty)}
    \) is unital.
\end{lemma}
\begin{proof}
    Let us show that
    \(
        \mathcal A
        = \langle
            x^k a
        \mid
            x \in \mathcal R,
            a \in \mathcal A
        \rangle
    \) for every
    \(
        k \geq 0
    \). In other words, for all
    \(
        k \geq 0
    \) there is
    \(
        T \geq 0
    \) such that for all
    \(
        n \geq 0
    \) there is
    \(
        m \geq 0
    \) with the following property: for every
    \(
        R \in \Ring_K
    \) the image
    \(
        s^m \mathcal A(R)^{(s^n)}
    \) of
    \(
        \mathcal A(R)^{(s^{n + m})}
        \to \mathcal A(R)^{(s^n)}
    \) is contained in
    \[
        \sum_{t = 1}^T
            (R^{(s^n)})^{[k]}\, \mathcal A(R)^{(s^n)}
        = s^{kn} \sum_{t = 1}^T
            (R^{[k]} \mathcal A(R))^{(s^n)}
        \leq \mathcal A(R)^{(s^n)},
    \]
    where
    \(
        B^{[k]} = \{b^k \mid b \in B\}
    \) for
    \(
        B \in \Rng_K
    \). We may take
    \(
        T = 1
    \) and
    \(
        m = kn
    \).
\end{proof}

\begin{lemma} \label{cozariski}
    Let
    \(
        \mathcal A \colon \Ring_K \to \Rng_K
    \) be a functor commuting with direct limits and with a structure of an \(\mathcal R\)-algebra. Suppose that
    \(
        \mathcal D(s)
        = \bigcup_{i = 1}^N \mathcal D(s_i)
    \). Then
    \(
        \mathcal A^{(s^\infty)}
    \) is the sum of the images of
    \(
        \mathcal A^{(s_i^\infty)}
        \to \mathcal A^{(s^\infty)}
    \) for \(1 \leq i \leq N\).
\end{lemma}
\begin{proof}
    First of all, \(s\) invertible in all
    \(
        K_{s_i}
    \), so without loss of generality
    \(
        s_i = s a_i
    \) for some
    \(
        a_i \in K
    \). On the other hand,
    \(
        \sum_i K_s s_i = K_s
    \), i.e.
    \(
        s^n \in \sum_i Ks_i
    \) for sufficiently large \(n\). We have
    \[
        \sum_i \Image(
            \mathcal A(R)^{(s_i^m)}
            \to \mathcal A(R)^{(s^m)}
        )
        = (\sum_i a_i^m \mathcal A(R))^{(s^m)}
        \supseteq \mathcal A(R)^{
            (s^{m + n \max(0, Nm - N + 1)})
        }. \qedhere
    \]
\end{proof}

Fix
\(
    s \in K
\). One can define \(A\)-points of pointed affine \(K_s\)-schemes of finite presentation for every
\(
    A \in \IPRng_{\mathbf D, s}
\). Every pointed affine \(K_s\)-scheme is of the type
\(
    X = \Spec(R)
\) for some
\(
    R \in \Ring_{K_s}
\) together with a homomorphism
\(
    \varepsilon \colon R \to K_s
\) of \(K_s\)-algebras. If \(X\) is of finite presentation, then up to isomorphism
\(
    R = K_s[x_1, \ldots, x_n] / (f_1, \ldots, f_m)
\) and
\(
    \varepsilon(x_i) = 0
\) for some
\(
    f_i
    \in (x_1, \ldots, x_n)
    \leqt K_s[x_1, \ldots, x_n]
\). If
\(
    X' = \Spec(R')
\) is another such scheme with
\(
    R'
    = K_s[x'_1, \ldots, x'_{n'}]
    / (f'_1, \ldots, f'_{m'})
\) and
\(
    h \colon X \to X'
\) is a morphism of pointed schemes, then \(h\) is given by polynomials
\(
    h_1, \ldots, h_{n'} \in (x_1, \ldots, x_n)
\) (the images of \(x'_i\), so
\(
    f'_i(h_1, \ldots, h_{n'}) \in (f_1, \ldots, f_m)
\)) uniquely determined modulo
\(
    (f_1, \ldots, f_m)
\). For any
\(
    A \in \IPRng_{\mathbf D, s}
\) and finitely presented pointed affine \(K_s\)-scheme
\(
    X = \Spec(R)
\) with the presentation as above we construct the object
\[
    X(A) = \{
        \vec x \in A^n
    \mid
        f_1(\vec x) = \ldots = f_m(\vec x) = 0
    \} \in \IP_{\mathbf D}
\]
of \(A\)-points of \(X\) (the right hand side is the limit of
\(
    A^n \xrightarrow f A^m \xleftarrow 0 \{*\}
\)). Clearly, \(X(A)\) is functorial on both \(X\) and \(A\). The functors \(X({-})\) preserve all finite limits and monomorphisms, also the functors \(({-})(A)\) preserve all finite limits. For example, if
\(
    A \in \Rng_K
\), then \(X(A)\) is just the kernel of
\(
    X(A \rtimes K) \to X(K)
\), i.e. the limit of
\(
    X(A \rtimes K) \to X(K) \leftarrow \{*\}
\).

\section{Elementary groups over ind-pro-algebras}

In this section \(G\) be a reductive group scheme over \(K\) with an isotropic pinning
\(
    (T, \Psi)
\) of rank \(\geq 2\) over \(K_s\) for some
\(
    s \in K
\). We also fix a category \(\mathbf D\) from the definition of
\(
    \IP_{\mathbf D}
\), e.g.
\(
    \mathbf D = \{*\}
\) or \(
    \mathbf D = \Ring_K^\fp
\). For any
\(
    A \in \IPRng_{\mathbf D, s}
\) we have the group objects \(G(A)\) and
\(
    P_\alpha(A)
\) in
\(
    \IP_{\mathbf D}
\) for all
\(
    \alpha \in \Psi
\). The \textit{unrelativized elementary group} is the group subobject
\[
    \elem_{G, T, \Psi}(A) = \langle
        t_\alpha(A)
    \mid
        \alpha \in \Psi
    \rangle \leq G(A)
\]
in
\(
    \IP_{\mathbf D}
\) constructed in lemma \ref{subgr-gen} (in this definition we do not use the rank condition on
\(
    (\Psi, T)
\)).

Fix a power idempotent object
\(
    R \in \IPRng_{\mathbf D, s}
\) and
\(
    A \in \IPRng_{\mathbf D, s}
\) with a unital action of \(R\) (so \(A\) is an
\(
    R \rtimes K_s^{\Ind}
\)-algebra with the unitality condition). The \textit{relative elementary group} is
\[
    \elem_{G, T, \Psi}(R, A)
    = \Ker(
        \elem_{G, T, \Psi}(A \rtimes R)
        \to \elem_{G, T, \Psi}(R)
    )
    \leq G(A).
\]

We say that a subset
\(
    \Sigma \subseteq \Psi
\) is \textit{saturated special closed} if it is an intersection of \(\Psi\) with a convex cone not containing opposite non-zero vectors. Let
\[
    \elem_{G, T, \Sigma}(A)
    = \langle
        \Image(t_\alpha)
    \mid
        \alpha \in \Sigma
    \rangle
    \leq \elem_{G, T, \Psi}(A)
\]
and
\[
    z_\Sigma
    \colon \elem_{G, T, -\Sigma}(R)
    \times \elem_{G, T, \Sigma}(A)
    \to \elem_{G, T, \Psi}(R, A),\,
    (g, h) \mapsto \up gh
\]
for such \(\Sigma\). Let also
\(
    \elem'_{G, T, \Psi}(R, A)
\) be the subgroup of \(G(A)\) generated by the morphisms
\[
    z_\alpha
    \colon P_{-\alpha}(R) \times P_\alpha(A)
    \to G(A),\,
    (r, a)
    \mapsto \up{t_{-\alpha}(r)}{t_\alpha(a)}.
\]

\begin{lemma} \label{rel-elem}
    \(
        \elem_{G, T, \Psi}(R, A)
        = \langle
            \up g{t_\alpha(a)}
        \mid
            g \in \elem_{G, T, \Psi}(R),
            a \in P_\alpha(A),
            \alpha \in \Psi
        \rangle
    \).
\end{lemma}
\begin{proof}
    Let \(H\) be the right hand side of the asserted equality. Clearly,
    \[
        \elem_{G, T, \Psi}(A)
        \leq H
        \leq \elem_{G, T, \Psi}(R, A).
    \]
    Also,
    \(
        H \rtimes \elem_{G, T, \Psi}(R)
        \subseteq G(A \rtimes R)
    \) is a subgroup (instead of just a subobject) since \(H\) is
    \(
        \elem_{G, T, \Psi}(R)
    \)-invariant. It follows that
    \(
        H \rtimes \elem_{G, T, \Psi}(R)
        = \elem_{G, T, \Psi}(A \rtimes R)
    \), so
    \(
        H = \elem_{G, T, \Psi}(R, A)
    \).
\end{proof}

\begin{lemma} \label{loc-root-gen}
    Let
    \(
        (T, \Psi)
    \) be an isotropic pinning of \(G\) of rank \(\geq 2\),
    \(
        \alpha \in \Psi
    \), \(e_k\) be a basis vector of \(P_\alpha\), and
    \(
        \alpha \in H \leq \mathbb R \Psi
    \) be a hyperplane. We denote the sets of roots of \(\Psi\) on the same side with respect to \(H\) by \(\Psi_1\) and \(\Psi_2\). Then there are
    \(
        V \geq 0
    \) and elements
    \(
        \beta_v \in \Psi_1
    \),
    \(
        \gamma_v \in \Psi_2
    \),
    \(
        p_v \in P_{\beta_v}(K)
    \),
    \(
        q_v \in P_{\gamma_v}(x K[x])
    \) for
    \(
        1 \leq v \leq V
    \) such that \(\beta_v\) are neighbors of \(\alpha\),
    \(
        i_v \beta_v + \gamma_v
        \in \{\alpha, 2 \alpha\} \cap \Psi
    \) for some
    \(
        i_v \geq 1
    \), and
    \[
        e_k \cdot x
        = \sum^\cdot_{1 \leq v \leq V}
            f_{\beta_v \gamma_v i_v 1}(p_v, q_v)
        \in P_\alpha(x K[x]).
    \]
\end{lemma}
\begin{proof}
    Recall that in the split case (i.e. if
    \(
        (T, \Psi)
    \) is a pinning)
    \(
        f_{\beta \gamma i 1}(x, y)
        = c_{\beta \gamma i 1} x^i y
    \),
    \(
        c_{\beta \gamma i 1} \in K^*
    \) if
    \(
        (\beta, \gamma)
    \) is a base of a saturated root subsystem (i.e. an intersection of \(\Psi\) and a plane) and
    \(
        f_{\beta \gamma 1 1}(x, y)
        = c_{\beta \gamma 1 1} x y
    \),
    \(
        c_{\beta \gamma 1 1} \in 3 K^*
    \) if \(\beta\), \(\gamma\) are short roots in a component of type
    \(
        \mathsf G_2
    \) and
    \(
        \beta + \gamma
    \) is a long root.

    Now consider the \(K\)-submodule (or \(2\)-step nilpotent \(K\)-submodule)
    \(
        N \leq P_\alpha(x K[x])
    \) generated by
    \(
        f_{\beta \gamma i 1}(
            P_\beta(K),
            P_\gamma(x K[x])
        )
    \), where
    \(
        i \beta + \gamma
        \in \{\alpha, 2 \alpha\} \cap \Psi
    \),
    \(
        i \in \{1, 2\}
    \), and
    \(
        \beta \in \Psi_1
    \) is a neighbor to \(\alpha\). It it easy to see that \(N\) is preserved under scalar extensions. In order to study \(N\) we may assume that
    \(
        (T, \Psi)
    \) is contained in a pinning
    \(
        (T', \Phi)
    \) of \(G\). By lemma \ref{root-sys-dec} any preimage
    \(
        \alpha' \in \Phi
    \) of \(\alpha\) or \(2 \alpha\) has a neighbor \(\beta'\) with the image in
    \(
        \Psi_1
    \), i.e. the subgroup
    \(
        P_{\alpha'}(x K[x]) \leq P_\alpha(x K[x])
    \) is contained in \(N\) if \(\alpha'\) is not a long root in a component of type
    \(
        \mathsf G_2
    \) and
    \(
        3 P_{\alpha'}(x K[x]) \leq N
    \) otherwise. So
    \(
        N = P_\alpha(x K[x])
    \) unless \(\alpha\) is long in a component of type
    \(
        \mathsf G_2
    \).

    In the exceptional case we have only
    \(
        3 P_\alpha(x K[x]) \leq N
    \). Consider the \(K\)-submodule
    \(
        N' \leq P_\alpha(x K[x])
    \) generated by
    \(
        f_{\beta \gamma i 1}(
            P_\beta(K),
            P_\gamma(x K[x])
        )
    \), where
    \(
        i \beta + \gamma
        \in \{\alpha, 2 \alpha\} \cap \Psi
    \),
    \(
        i \geq 1
    \) is arbitrary, and
    \(
        \beta \in \Psi_1
    \) is a neighbor to \(\alpha\). The submodule \(N'\) is also preserved under scalar extensions since it contains
    \(
        3 P_\alpha(x K[x])
    \). The same argument as above implies that
    \(
        N' = P_\alpha(x K[x])
    \).
\end{proof}

\begin{lemma} \label{root-gen}
    For any
    \(
        \alpha \in \Psi
    \) and a hyperplane
    \(
        \alpha \in H \leq \mathbb R \Psi
    \) the group object
    \(
        P_\alpha(A)
    \) is generated by
    \(
        f_{\beta \gamma i 1}(P_\beta(R), P_\gamma(A))
    \) for all
    \(
        \beta, \gamma \in \Psi
    \) and
    \(
        i \geq 1
    \) such that
    \(
        i \beta + \gamma
        \in \{\alpha, 2 \alpha\} \cap \Psi
    \) and
    \(
        \beta \in \Psi_1
    \) is a neighbor of \(\alpha\), where \(\Psi_1\) and \(\Psi_2\) are the sets of roots of \(\Psi\) on the same side with respect to \(H\).
\end{lemma}
\begin{proof}
    By lemma \ref{loc-root-gen} there are identities
    \[
        e_k \cdot ar^3
        = \sum^\cdot_{1 \leq v \leq V_k}
            f_{\beta_{kv}, \gamma_{kv}, i_{kv}, 1}(
                p_{kv} \cdot r,
                q_{kv}(a) \cdot r^{3 - i_{kv}}
            )
    \]
    of morphisms
    \(
        A \times R \to P_\alpha(A)
    \), where
    \(
        \beta_{kv} \in \Psi_1
    \) are neighbors of \(\alpha\),
    \(
        i_{kv} \geq 1
    \),
    \(
        i_{kv} \beta_{kv} + \gamma_{kv}
        \in \{\alpha, 2 \alpha\} \cap \Psi
    \),
    \(
        p_{kv} \in P_{\beta_{kv}}(K_s)
    \), and
    \(
        q_{kv} \in P_{\gamma_{kv}}(x K_s[x])
    \).
\end{proof}

\begin{lemma} \label{elim-abs}
    For any
    \(
        \alpha \in \Psi
    \) the elementary subgroup
    \(
        \elem_{G, T, \Psi}(R)
    \) is generated by
    \(
        t_\beta(R)
    \) for
    \(
        \beta \in \Phi \setminus \mathbb R \alpha
    \).
\end{lemma}
\begin{proof}
    This easily follows from lemma \ref{root-gen} applied to
    \(
        A = R
    \).
\end{proof}

\begin{lemma} \label{z-sigma}
    For any saturated special closed subset
    \(
        \Sigma \subseteq \Psi
    \) the morphism \(z_\Sigma\) takes values in
    \(
        \elem'_{G, T, \Psi}(R, A)
    \).
\end{lemma}
\begin{proof}
    We use induction on \(|\Sigma|\), the case
    \(
        \dim(\mathbb R \Sigma) \leq 1
    \) is clear. If
    \(
        \Sigma \neq \varnothing
    \), then there exists an \textit{extreme root}
    \(
        \alpha \in \Sigma
    \), i.e. a root such that
    \(
        \frac 1 2 \alpha \notin \Sigma
    \) and
    \(
        \mathbb R_{\geq 0} \alpha
    \) is an extreme ray of the convex cone spanned by \(\Sigma\) (then
    \(
        \Sigma \setminus \mathbb R \alpha
    \) is also a saturated special closed set). Moreover, if
    \(
        \dim(\mathbb R \Sigma) \geq 2
    \), then there is an extreme root
    \(
        \alpha \in \Sigma
    \) linearly independent with any given element of \(\Sigma\).

    Since \(z_\Sigma\) is a homomorphism on the second argument, it suffice to check that any
    \(
        z_\Sigma(g, t_\alpha(a))
    \) lie in
    \(
        \elem'_{G, T, \Psi}(R, A)
    \). Take an extreme
    \(
        \beta \in \Sigma
    \) linearly independent with \(\alpha\), so the product map
    \[
        \elem_{
            G, T, \Sigma \setminus \mathbb R \beta
        }(R)
        \times P_\beta(R)
        \to \elem_{G, T, \Sigma}(R)
    \]
    is an isomorphism. Now
    \[
        z_\Sigma(g t_\beta(r), t_\alpha(a))
        = z_{\Sigma \setminus \mathbb R \beta}(
            g,
            \up{t_\beta(r)}{t_\alpha(a)}
        )
    \]
    for
    \(
        g \in \elem_{
            G, T, \Sigma \setminus \mathbb R \beta
        }(R)
    \) and the right hand side lies in
    \(
        \elem'_{G, T, \Psi}(R, A)
    \) by the induction hypothesis.
\end{proof}

Let
\(
    \alpha \in \Psi
\). We say that
\(
    \Sigma \subseteq \Psi
\) is a \textit{thick \(\alpha\)-series} if is if of type
\(
    \Sigma = \Psi \cap (
        \mathbb R_{> 0} \beta
        + \mathbb R \alpha
    )
\) for some
\(
    \beta \in \Psi \setminus \mathbb R \alpha
\). Clearly,
\(
    \Psi \setminus \mathbb R \alpha
\) is a disjoint decomposition of thick \(\alpha\)-series. Every thick \(\alpha\)-series is a saturated special closed set.

\begin{lemma} \label{elim-rel}
    For any
    \(
        \alpha \in \Psi
    \) the group object
    \(
        \elem'_{G, T, \Psi}(R, A)
    \) is generated by the images of \(z_\Sigma\) for thick \(\alpha\)-series \(\Sigma\).
\end{lemma}
\begin{proof}
    Without loss of generality,
    \(
        \frac 1 2 \alpha \notin \Psi
    \). Since
    \(
        t_{\pm \alpha}(P_{\pm \alpha}(R))
    \) normalize
    \(
        \elem_{G, T, \Sigma}(A)
    \) for all thick \(\alpha\)-series \(\Sigma\), we only have to check that
    \(
        t_{\pm \alpha}(a)
    \) for
    \(
        a \in P_{\pm \alpha}(A)
    \) may be expressed in terms of these \(z_\Sigma\). This easily follows from lemma \ref{root-gen}.
\end{proof}

\begin{lemma} \label{expl-conj}
The subgroup
\(
    \elem_{G, T, \Psi}(R)
\) normalizes
\(
    \elem^\prime_{G, T, \Psi}(R, A) \leq G(A)
\).
\end{lemma}
\begin{proof}
    It suffices to check that for any root \(\alpha\) the subgroup
    \(
        t_\alpha(P_\alpha(R))
    \) normalizes
    \(
        \elem'_{G, T, \Psi}(R, A)
    \). This follows from lemma \ref{elim-rel}.
\end{proof}

\begin{lemma} \label{rel-expl}
    \(
        \elem^\prime_{G, T, \Psi}(R, A)
        = \elem_{G, T, \Psi}(R, A)
        \leq G(A)
    \).
\end{lemma}
\begin{proof}
    Clearly, the first subgroup is contained in the second one. Conversely, by lemma \ref{expl-conj} all generators of
    \(
        \elem_{G, T, \Psi}(R, A)
    \) from lemma \ref{rel-elem} lie in
    \(
        \elem'_{G, T, \Psi}(R, A)
    \).
\end{proof}

\begin{lemma} \label{rel-to-abs}
    If \(A\) is also power idempotent then
    \(
        \elem_{G, T, \Psi}(A)
        = \elem_{G, T, \Psi}(R, A)
    \).
\end{lemma}
\begin{proof}
    By lemma \ref{rel-expl} it suffices to check that for any
    \(
        \alpha \in \Psi
    \) the subgroup
    \(
        \elem_{G, T, \Psi}(A)
    \) is normalized by
    \(
        t_\alpha(P_\alpha(R))
    \). This easy follows from lemma \ref{elim-abs}.
\end{proof}

\begin{lemma} \label{pin-elim}
    If
    \(
        (T', \Psi') \subseteq (T, \Psi)
    \) is another isotropic pinning, then
    \(
        \elem_{G, T', \Psi'}(R)
        \leq \elem_{G, T, \Psi}(R)
    \). If the rank of \((T', \Psi')\) is positive, then
    \(
        \elem_{G, T', \Psi'}(R)
        = \elem_{G, T, \Psi}(R)
    \).
\end{lemma}
\begin{proof}
    Clearly,
    \(
        \elem_{G, T', \Psi'}(R)
        \leq \elem_{G, T, \Psi}(R)
    \). If
    \(
        \alpha \in \Psi
    \) has non-zero image
    \(
        \alpha' \in \Psi'
    \), then \(t_\alpha\) factors through \(t_{\alpha'}\). Finally, if
    \(
        \alpha \in \Psi
    \) has zero image in
    \(
        \Psi' \cup \{0\}
    \), then
    \(
        t_\alpha(P_\alpha(R))
        \leq \elem_{G, T', \Psi'}(R)
    \) by lemma \ref{root-gen}.
\end{proof}

\begin{lemma} \label{xmod-dec}
    Suppose that
    \(
        \delta \colon A \to R
    \) is a crossed module in
    \(
        \IPRng_{\mathbf D, s}
    \) and
    \(
        A = \sum_{i = 1}^n A_i
    \) for crossed
    \(
        (R \rtimes K_s^{\Ind})
    \)-submodules \(A_i\) unital over \(R\). Then
    \(
        \elem_{G, T, \Psi}(R, A)
        = \prod_{i = 1}^n \elem_{G, T, \Psi}(R, A_i)
    \). If
    \(
        P \leq G_s
    \) is a parabolic subgroup, then also
    \(
        \uradical(P)(A)
        = \prod_{i = 1}^n \uradical(P)(A_i)
    \), where
    \(
        \uradical(P)
    \) is the unipotent radical of \(P\).
\end{lemma}
\begin{proof}
    Note that
    \(
        \elem_{G, T, \Psi}(R, A_i)
        \leqt \elem_{G, T, \Psi}(R, A)
    \) since this subgroup is normalized by
    \(
        \elem_{G, T, \Psi}(R)
    \), so the Minkowski product
    \(
        \prod_{i = 1}^n \elem_{G, T, \Psi}(R, A_i)
    \) is a subgroup of
    \(
        \elem_{G, T, \Psi}(A)
    \). On the other hand,
    \(
        t_\alpha(P_\alpha(A))
        \leq \prod_{i = 1}^n
            t_\alpha(P_\alpha(A_i))
    \) for any \(\alpha \in \Psi\).

    For unipotent radicals there is a finite nilpotent filtration
    \(
        \uradical(P) = U_1
        \geq U_2
        \geq \ldots
        \geq U_m = 0
    \) such that
    \(
        U_k \leq \uradical(P)
    \) are closed smooth \(P\)-invariant subschemes,
    \(
        [U_k, U_l] \leq U_{k + l}
    \) (i.e. the commutator restricted to
    \(
        U_k \times U_l
    \) takes values in
    \(
        U_{k + l}
    \)), and
    \(
        U_k / U_{k + 1}
    \) is isomorphic to a vector bundle over
    \(
        \mathcal D(s)
    \) \cite[Exp. XXVI, proposition 2.1]{sga3}. By \cite[Exp. XXVI, the proof of corollary 2.5]{sga3} the morphisms
    \(
        U_k \to U_k / U_{k + 1}
    \) have scheme sections, so
    \(
        (U_k / U_{k + 1})(A)
        \cong U_k(A) / U_{k + 1}(A)
    \) for every \(k\). It remains to prove that
    \(
        V(A) = \sum_i V(A_i)
    \) for any vector bundle \(V\) over
    \(
        \mathcal D(s)
    \). In other words, if
    \(
        e \in \mat(N, K_s)
    \) is an idempotent for some
    \(
        N \geq 0
    \), then
    \(
        e(A^N) = \sum_i e(A_i^N)
    \). This easily follows from
    \(
        A = \sum_i A_i
    \).
\end{proof}

\section{Locally isotropic elementary groups}

In this section \(G\) is a reductive group scheme over \(K\) of local isotropic rank at least \(2\). Choose a Zariski covering
\(
    \Spec(K) = \bigcup_{i = 1}^N \mathcal D(a_i)
\) for some
\(
    a_i \in K
\) (i.e. \(a_i\) generate the unit ideal) and isotropic pinnings
\(
    (T_i, \Psi_i)
\) of
\(
    G_{a_i}
\) of rank \(\geq 2\). Such data exist since the isotropic rank of \(G\) is at least \(2\) at every point of
\(
    \Spec(K)
\) and an isotropic pinning of
\(
    G_{\mathfrak p}
\) for a prime ideal
\(
    \mathfrak p \leqt K
\) may be continued to \(G_s\) for some
\(
    s \in K \setminus \mathfrak p
\). For any unital \(K\)-algebra \(R\) the \textit{elementary subgroup} is
\[
    \elem_G(R) = \langle
        \Image(
            \elem_{G, T_i, \Psi_i}(R^{(a_i^\infty)})
            \to G(R)
        )
    \mid
        1 \leq i \leq N
    \rangle,
\]
it is a group object in
\(
    \IP_{\{*\}} = \Ind(\Pro(\Set))
\). We are going to show that it is independent on the choices and is actually a normal subgroup of \(G(R)\). In order to study functorial properties of
\(
    \elem_G
\) let also
\[
    \elem_G(\mathcal R) = \langle
        \Image(
            \elem_{G, T_i, \Psi_i}(
                \mathcal R^{(a_i^\infty)}
            ) \to G(\mathcal R)
        )
    \mid
        1 \leq i \leq N
    \rangle
\]
as a group object in
\(
    \IP_{\Ring_K^\fp}
\), so \(\elem_G(R)\) is the restriction of
\(
    \elem_G(\mathcal R)
\) to not necessarily finitely presented \(R\).

\begin{lemma} \label{levi-action}
    Let
    \(
        s \in K
    \),
    \(
        (T, \Psi)
    \) be an isotropic pinning of \(G_s\), and \(L\) be the centralizer of \(T\) in \(G_s\). Then the subgroup
    \(
        \gelem_{G, T, \Psi}(\mathcal R_s^{\Ind})
        = L(\mathcal R_s^{\Ind})
        \elem_{G, T, \Psi}(\mathcal R_s^{\Ind})
        \leq G(\mathcal R_s^{\Ind})
    \) normalizes
    \(
        \elem_{G, T, \Psi}(\mathcal R^{(s^\infty)})
    \) and
    \(
        \elem_{G, T, \Psi}(
            \mathcal R^{(s^\infty)},
            \mathcal A^{(s^\infty)}
        )
    \) for any functor
    \(
        \mathcal A \colon \Ring_K \to \Rng_K
    \) commuting with direct limits and with a structure of an \(\mathcal R\)-algebra.
\end{lemma}
\begin{proof}
    For
    \(
        t_\alpha(P_\alpha(\mathcal R^{\Ind}_s))
    \) the claim follows from lemmas \ref{elim-abs}, \ref{elim-rel}, \ref{rel-expl}. For every
    \(
        \alpha \in \Psi
    \) and a basis vector
    \(
        e_k \in P_\alpha
    \) there is a scheme morphism
    \(
        h_{\alpha k}
        \colon L \times \mathbb A^1
        \to P_\alpha
    \) such that
    \[
        \up g{t_\alpha(e_k \cdot r)}
        = t_\alpha(h_{\alpha k}(g, r)).
    \]
    Moreover,
    \(
        h_{\alpha k}(g, 1) = 1
    \). Taking limits of the columns of the commutative diagram
    \[\xymatrix@R=30pt@C=150pt@!0{
        L(\mathcal R^{\Ind}_s)
        \times (
            \mathcal A^{(s^\infty)}
            \rtimes \mathcal R^{\Ind}_s
        )
        \ar^(0.55){h_{\alpha k}}[r]
        \ar[d]
        &
        P_\alpha(
            \mathcal A^{(s^\infty)}
            \rtimes \mathcal R^{\Ind}_s
        )
        \ar[d]
        \\
        L(\mathcal R^{\Ind}_s)
        \times \mathcal R^{\Ind}_s
        \ar^(0.55){h_{\alpha k}}[r]
        &
        P_\alpha(\mathcal R^{\Ind}_s)
        \\
        L(\mathcal R^{\Ind}_s)
        \times \{*\}
        \ar[r]
        \ar[u]
        &
        \{*\},
        \ar[u]
    }\]
    we obtain a morphism
    \(
        h_{\alpha k}
        \colon L(\mathcal R_s^{\Ind})
        \times \mathcal A^{(s^\infty)}
        \to P_\alpha(\mathcal A^{(s^\infty)})
    \) such that
    \[
        \up g{t_\alpha(e_k \cdot a)}
        = t_\alpha(h_{\alpha k}(g, a))
    \]
    for
    \(
        g \in L(\mathcal R_s^{\Ind})
    \) and
    \(
        a \in \mathcal A^{(s^\infty)}
    \).
\end{proof}

The next lemma shows that
\(
    \elem_G(\mathcal R)
\) is independent of the Zariski covering and the isotropic pinnings
\(
    (T_i, \Psi_i)
\). Moreover, for any
\(
    R \in \Ring_K
\) and
\(
    g \in \Aut(G)(R)
\) the subgroup
\(
    \elem_G(R) \leq G(R)
\) is \(g\)-invariant (this does not imply the normality since
\(
    \elem_G(R)
\) is a group object in
\(
    \Ind(\Pro(\Set))
\) rather than \(\Set\)). We denote the unipotent radical of a parabolic subgroup
\(
    P \leq G
\) by
\(
    \uradical(P)
\).

\begin{lemma} \label{invariance}
    Let
    \(
        s \in K
    \) and \(P\) be a parabolic subgroup of \(G_s\). Then the image of
    \(
        \uradical(P)(\mathcal R^{(s^\infty)})
    \) in
    \(
        G(\mathcal R)
    \) is contained in
    \(
        \elem_G(\mathcal R)
    \). In particular, if
    \(
        (T, \Psi)
    \) is an isotropic pinning of \(G_s\), then the image of
    \(
        \elem_{G, T, \Psi}(\mathcal R^{(s^\infty)})
    \) in
    \(
        G(\mathcal R)
    \) is contained in
    \(
        \elem_G(\mathcal R)
    \).
\end{lemma}
\begin{proof}
    For all
    \(
        1 \leq i \leq N
    \) choose Zariski coverings
    \(
        \mathcal D(a_i) \cap \mathcal D(s)
        = \bigcup_{j = 1}^{M_i} \mathcal D(b_{ij})
    \), isotropic pinnings
    \(
        (T_{ij}, \Psi_{ij})
    \) of
    \(
        G_{b_{ij}}
    \) containing
    \(
        (T_i|_{\mathcal D(b_{ij})}, \Psi_i)
    \), and elements
    \(
        g_{ij}
        \in \elem_{G, T_{ij}, \Psi_{ij}}(K_{b_{ij}})
    \) such that
    \(
        P|_{\mathcal D(b_{ij})}
        \supseteq \up{g_{ij}}{P_{ij}}
    \), where
    \(
        P_{ij}
    \) is the parabolic subgroup corresponding to
    \(
        (T_{ij}, \Psi_{ij})
    \). Such data exist by lemma \ref{loc-iso-pin}. By lemmas \ref{pin-elim} and \ref{levi-action} we have
    \[
        \uradical(P)(\mathcal R^{(b_{ij}^\infty)})
        \leq \elem_{
            G,
            \up{g_{ij}}{T_{ij}},
            \up{g_{ij}}{\Psi_{ij}}
        }(
            \mathcal R^{(b_{ij}^\infty)}
        )
        = \elem_{G, T_{ij}, \Psi_{ij}}(
            \mathcal R^{(b_{ij}^\infty)}
        )
        = \elem_{G, T_i, \Psi_i}(
            \mathcal R^{(b_{ij}^\infty)}
        )
    \]
    (the left inclusion is obvious). From lemmas \ref{cozariski} and \ref{xmod-dec} we obtain
    \[
        \Image(\elem_{G, T, \Psi}(
            \mathcal R^{(s^\infty)}
        ))
        \leq \elem_G(\mathcal R). \qedhere
    \]
\end{proof}

\begin{lemma} \label{loc-gluing}
    Let
    \(
        (T, \Psi)
    \) and
    \(
        (T', \Psi')
    \) be isotropic pinnings of \(G\) such that the rank of
    \(
        (T', \Psi')
    \) is \(\geq 2\). Then there is
    \(
        n \geq 1
    \) such that
    \(
        \elem_{G, T, \Psi}(x^n K[x])
        \leq \elem_{G, T', \Psi'}(x K[x])
    \).
\end{lemma}
\begin{proof}
    By lemmas \ref{rel-to-abs} (applied to
    \(
        R = K[x]
    \) and
    \(
        A
        = \varprojlim^{\Pro(\Set)}_n x^n K[x]
        \cong K[x]^{(x^\infty)}
    \)) and \ref{invariance} we have
    \begin{align*}
        \elem_{G, T, \Psi}(x^n K[x])
        &\leq \elem_{G, T, \Psi}(K[x], x^n K[x])
        \leq \Ker(
            \elem_G(x^n K[x] \rtimes K[x])
            \to \elem_G(K[x])
        )\\
        &= \elem_{G, T', \Psi'}(K[x], x^n K[x])
        \leq \elem_{G, T', \Psi'}(x K[x])
    \end{align*}
    for sufficiently large \(n\).
\end{proof}

\begin{lemma} \label{main-lemma}
    Let
    \(
        s \in K
    \) and
    \(
        (T, \Psi)
    \) be an isotropic pinning of \(G_s\). Then there is
    \(
        m \geq 0
    \) such that for any
    \(
        \alpha \in \Psi
    \) the image of
    \(
        t_\alpha(P_\alpha(\mathcal R^{(s^m)}))
    \) is contained in
    \(
        \elem_G(\mathcal R)
    \) (such a subset is well-defined for all sufficiently large \(m\)).
\end{lemma}
\begin{proof}
    Fix a root \(\alpha\) and a basis vector
    \(
        e_k \in P_\alpha
    \). Choose a Zariski covering
    \(
        \Spec(K) = \bigcup_{i = 1}^N \mathcal D(a_i)
    \) and isotropic pinnings
    \(
        (T_i, \Psi_i)
    \) of
    \(
        G_{a_i}
    \) of rank \(\geq 2\). By lemma \ref{loc-gluing} applied to
    \(
        K_{s a_i}[z]
    \) there are
    \(
        n \geq 0
    \),
    \(
        M_i \geq 0
    \), roots
    \(
        \beta_{ij} \in \Psi_i
    \) for
    \(
        1 \leq j \leq M_i
    \), and elements
    \(
        p_{ij} \in P_{\beta_{ij}}(x K_{s a_i}[x, z])
    \) such that the image of
    \(
        t_\alpha(e_k \cdot x^n z)
    \) in
    \(
        G(K_{s a_i}[x, z])
    \) is
    \(
        \prod_{j = 1}^{M_i} t_{\beta_{ij}}(p_{ij})
    \). This may be considered as an equality between two scheme morphisms
    \(
        \mathbb A^1_{K_{sa_i}[z]}
        \rightrightarrows G_{K_{sa_i}[z]}
    \) over
    \(
        K_{a_i}[z]
    \), where \(x\) is the coordinate on the affine line. Applying the equality to the
    \(
        (K_{a_i}[y, z])^{(s^\infty)}
    \)-points, we obtain that the following diagram is commutative:
    \[\xymatrix@R=60pt@C=90pt@!0{
        &
        G_s((K_{a_i}[y, z])^{(s^\infty)})
        \ar[dr]
        \\
        (K_{a_i}[y, z])^{(s^\infty)}
        \ar^{
            x
            \mapsto t_\alpha(e_k \cdot x^n z)
        }[ur]
        \ar_{
            x
            \mapsto \prod_{j = 1}^{M_i}
                t_{\beta_{ij}}(\Image(p_{ij}(x, z)))
        }[rr]
        &&
        G(K_{a_i}[y, z]).
    }\]
    The paths of this diagram are well-defined and coincide on
    \(
        (K_{a_i}[y, z])^{(s^{m_i})}
    \) for sufficiently large \(m_i\). Evaluating them on
    \(
        y^{(s^{m_i})}
    \) we obtain
    \[
        \Image(
            t_\alpha(e_k \cdot y^n z^{(s^{m_i n})})
        )
        = \prod_{j = 1}^{M_i} t_{\beta_{ij}}(q_{ij})
        \in G(K_{a_i}[y, z])
    \]
    for some
    \(
        q_{ij} \in P_{\beta_{ij}}(y K_{a_i}[y, z])
    \) (namely, such that
    \(
        p_{ij}(s^{m_i} y, z)
        = \Image(q_{ij})
        \in P_{\beta_{ij}}(y K_{sa_i}[y, z])
    \)). This is an equality between two scheme morphisms
    \(
        \mathbb A^2_{a_i} \to G_{a_i}
    \). In particular, for
    \(
        \mathcal R^{(a_i^\infty)}
    \)-points we have
    \(
        \Image(t_\alpha(
            e_k \cdot (
                \mathcal R^{(s^{nm_i})}
            )^{(a_i^\infty)}
        ))
        \leq \elem_G(\mathcal R)
    \) by lemma \ref{power-idem}, so
    \(
        \Image(t_\alpha(e_k \cdot \mathcal R^{(s^m)}))
        \leq \elem_G(\mathcal R)
    \) for sufficiently large \(m\) by lemma \ref{cozariski}.
\end{proof}

\begin{theorem} \label{e-discr}
    Let \(G\) be a reductive group scheme of local isotropic rank \(\geq 2\) over a unital ring \(K\). Then the elementary group
    \(
        \elem_G(\mathcal R) \leq G(\mathcal R)
    \) is isomorphic to an object from
    \(
        \Ind(\Cat(\Ring_K^\fp, \Set))
        \subseteq \IP_{\Ring_K^\fp}
    \). More precisely, if
    \(
        \Spec(K) = \bigcup_{i = 1}^N \mathcal D(a_i)
    \) is a Zariski covering and
    \(
        (T_i, \Psi_i)
    \) are isotropic pinnings of
    \(
        G_{a_i}
    \) of rank \(\geq 2\), then there is
    \(
        m \geq 0
    \) such that the maps
    \(
        t_\alpha
        \colon P_\alpha(\mathcal R^{(a_i^m)})
        \to G(\mathcal R)
    \) are well-defined and generate
    \(
        \elem_G(\mathcal R)
    \) for
    \(
        \alpha \in \Psi_i
    \),
    \(
        1 \leq i \leq N
    \). Moreover,
    \(
        P_\alpha(\mathcal R^{(a_i^m)})
    \) are \(\mathcal R\)-modules (\(2\)-step nilpotent if \(\alpha\) is ultrashort),
    \(
        t_\alpha
    \) are homomorphisms, and
    \(
        f_{\alpha \beta i j}
        \colon P_\alpha(\mathcal R^{(a_i^m)})
        \times P_\beta(\mathcal R^{(a_i^m)})
        \to P_{i \alpha + j \beta}(\mathcal R^{(a_i^m)})
    \) are homogeneous polynomial maps over \(\mathcal R\) (or \(\mathcal R\)-quadratic on each argument for components of type
    \(
        \mathsf{BC}_\ell
    \)). If \(K\) is Noetherian, then for any fixed algebra
    \(
        R \in \Ring_K^\fp
    \) we may change the homotopes to the corresponding principal ideals.
\end{theorem}
\begin{proof}
    By lemma \ref{main-lemma} there is \(m\) such that \(
        t_\alpha
        \colon P_\alpha(\mathcal R^{(a_i^m)})
        \to G(\mathcal R)
    \) are well-defined and take values in
    \(
        \elem_G(\mathcal R)
    \). On the other hand, the generating morphisms \(
        t_\alpha
        \colon P_\alpha(\mathcal R^{(a_i^\infty)})
        \to G(\mathcal R)
    \) clearly factors through \(
        t_\alpha
        \colon P_\alpha(\mathcal R^{(a_i^m)})
        \to G(\mathcal R)
    \).

    The remaining assertions are easy. The \(2\)-step nilpotent \(\mathcal R\)-module structure on
    \(
        P_\alpha(\mathcal R^{(a_i^m)})
    \) is well-defined for ultrashort \(\alpha\) and sufficiently large \(m\) since it is given by some
    \(
        K_{a_i}
    \)-bilinear \(2\)-cocycle. Actually, every polynomial
    \(
        f \in K_{a_i}[x_1, \ldots, x_p]
    \) without terms of degree \(\leq 1\) induces a well-defined map
    \(
        (\mathcal R^{(a_i^m)})^p
        \to \mathcal R^{(a_i^m)}
    \) if \(m\) is large enough. The last claim follows from lemma \ref{noeth-coloc}.
\end{proof}

By theorem \ref{e-discr} the elementary group
\(
    \elem_G(R)
\) is a group object in
\(
    \Ind(\Set)
\) for every
\(
    R \in \Ring_K
\). We may consider it as an ordinary subgroup of \(G(R)\) by evaluating the direct limit in \(\Set\) (i.e. by taking the union of components of
\(
    \elem_G(R)
\)).

Let us say that a subfunctor
\(
    H \colon \Ring_K \to \Group
\) of \(G\) is \textit{scheme generated} if there is an affine \(K\)-scheme \(X\) and a morphism
\(
    f \colon X \to G
\) such that
\(
    H(R) = \langle f(X(R)) \rangle
\) as an abstract group for every
\(
    R \in \Ring_K
\) (here \(G\) is any affine group \(K\)-scheme, not necessarily reductive). The next lemma shows that such subfunctor has a canonical ind-structure.

\begin{lemma} \label{sch-gen}
    Let
    \(
        H \colon \Ring_K \to \Group
    \) be a subgroup of \(G\). If it is scheme generated by
    \(
        f \colon X \to G
    \) and a scheme morphism
    \(
        f' \colon X' \to G
    \) takes values in \(H\) for an affine \(K\)-scheme \(X'\), then there is
    \(
        n \geq 0
    \) such that
    \[
        f'(X'(R))
        \subseteq (
            f(X(R)) \cup \{1\} \cup f(X(R))^{-1}
        )^n
    \] for all
    \(
        R \in \Ring_K
    \). In particular, if \(f'\) also scheme generates \(H\), then the subgroups of \(G\) in
    \(
        \Ind(\Cat(\Ring_K^\fp, \Set))
    \) generated by
    \(
        \Image(f)
    \) and
    \(
        \Image(f')
    \) coincide.
\end{lemma}
\begin{proof}
    Let
    \(
        \widehat{X'}
    \) be the coordinate \(K\)-algebra of \(X'\) and
    \(
        x'_0 \in X'(\widehat{X'})
    \) be the universal point corresponding to
    \(
        \id \colon X' \to X'
    \) by the Yoneda lemma. By definition, there are
    \(
        n \geq 0
    \),
    \(
        a_i \in X(\widehat{X'})
    \), and
    \(
        \eps_i \in \{-1, 1\}
    \) for
    \(
        1 \leq i \leq n
    \) such that
    \(
        f'(x'_0) = \prod_{i = 1}^n f(a_i)^{\eps_i}
    \). If
    \(
        x' \in X'(R)
    \) is any \(R\)-point, then
    \(
        x' = u(x'_0)
    \) for unique homomorphism
    \(
        u \colon \widehat{X'} \to R
    \) and
    \[
        f'(x')
        = u_*(f'(x'_0))
        = \prod_{i = 1}^n
            f(u_*(a_i))^{\eps_i}. \qedhere
     \]
\end{proof}

\begin{theorem} \label{elem-gen}
    Let \(G\) be a reductive group scheme of local isotropic rank \(\geq 2\) over a unital ring \(K\). Then the elementary subgroup
    \(
        \elem_G \leq G
    \) is scheme generated by a morphism
    \(
        f \colon \mathbb A^N \to G
    \) from an affine space and its ind-structure is induced by \(f\). In particular,
    \(
        \elem_G(R \times R')
        = \elem_G(R) \times \elem_G(R')
    \),
    \(
        \elem_G(R / I)
        = \Image(\elem_G(R) \to G(R / I))
    \) for any ideal
    \(
        I \leqt R
    \), and
    \(
        \elem_G
    \) commutes with direct limits as an abstract group. On the other hand,
    \(
        \elem_G
    \) is also scheme generated by a \(G\)-equivariant morphism
    \(
        u \colon X \to G
    \) from a finitely presented affine \(K\)-scheme \(X\) with an action of \(G\), so
    \(
        \elem_G \leq G
    \) is normal both as an abstract group subfunctor and as a group object in
    \(
        \IP_{\Ring_K^\fp}
    \).
\end{theorem}
\begin{proof}
    The first claim follows from theorem \ref{e-discr}. In order to prove the second claim it suffice to check that
    \[
        u \colon G \times \mathbb A^N \to G,\,
        (g, \vec x) \mapsto \up g{f(\vec x)}
    \]
    takes values in
    \(
        \elem_G
    \) as an abstract subgroup functor of \(G\). Actually, we only have to check that
    \(
        u(g_0, \vec x_0)
        = \up{g_0}{f(\vec x_0)} \in G(R)
    \), where
    \(
        R = \widehat G[x_1, \ldots, x_N]
    \), \(\widehat G\) is the coordinate algebra of \(G\), and
    \(
        (g_0, \vec x_0) \in G(R) \times \mathbb A^N(R)
    \) is the universal point. This follows from lemma \ref{invariance}.
\end{proof}

The next theorem shows that our elementary subgroups coincide with the elementary subgroups constructed in \cite{petrov-stavrova} if the latter ones are defined.

\begin{theorem} \label{par-gen}
    Let \(G\) be a reductive group scheme of local isotropic rank \(\geq 2\) over a unital ring \(K\) and
    \(
        P^+, P^- \leq G
    \) be opposite parabolic subgroups. Suppose that their unipotent radicals
    \(
        \uradical(P^\pm)
    \) non-trivially intersect all simple factors of
    \(
        G / \Cent(G)
    \) over geometric points. Then
    \(
        \elem_G
    \) is scheme generated by the subscheme
    \(
        \uradical(P^-) \times \uradical(P^+)
        \subseteq G
    \).
\end{theorem}
\begin{proof}
    The unipotent radicals are contained in
    \(
        \elem_G
    \) by lemma \ref{invariance}. Conversely, by lemma \ref{loc-iso-pin} there is a Zariski covering
    \(
        \Spec(K) = \bigcup_{i = 1}^N \mathcal D(a_i)
    \) and isotropic pinnings
    \(
        (T_i, \Psi_i)
    \) of \(G_s\) such that
    \(
        P_{a_i}^+
    \) contain the standard parabolic subgroups
    \(
        P_i \leq G_{a_i}
    \) constructed by
    \(
        (T_i, \Psi_i)
    \). By \cite[lemma 12]{petrov-stavrova} the subfunctor
    \(
        \elem_{G, T_i, \Psi_i} \leq G_{a_i}
    \) is generated by
    \(
        \uradical(P^{\pm})_{a_i}
    \). It follows that the image of
    \(
        \elem_{G, T_i, \Psi_i}(
            \mathcal R^{(a_i^\infty)}
        )
    \) in
    \(
        G(\mathcal R)
    \) is contained in the subgroup generated by
    \(
        \uradical(P^\pm)(\mathcal R)
    \).
\end{proof}

Finally, let us prove that
\(
    \elem_G
\) is perfect. It is well-known that
\(
    \elem(\mathsf B_2, \mathbb F_2)
\) and
\(
    \elem(\mathsf G_2, \mathbb F_2)
\) are not perfect even in the adjoint case (their derived subgroups are simple but of index \(2\)). Following \cite{luzgarev-stavrova} we impose the following additional condition: if
\(
    \mathfrak m \leqt K
\) is a maximal ideal such that
\(
    K / \mathfrak m \cong \mathbb F_2
\), then the geometric fiber of
\(
    G / \Cent(G)
\) at
\(
    K / \mathfrak m
\) has no simple factors of the types
\(
    \mathsf B_2
\) and
\(
    \mathsf G_2
\). Since the residue fields at non-maximal prime ideals of \(K\) are infinite, this condition also holds for all unital \(K\)-algebras.

\begin{theorem} \label{perfect}
    Let \(G\) be a reductive group scheme of local isotropic rank \(\geq 2\) over a unital ring \(K\). Suppose that the additional condition holds, i.e. the geometric fibers of
    \(
        G / \Cent(G)
    \) over all residue fields of \(2\) elements does not have components with root systems of the types
    \(
        \mathsf B_2
    \) and
    \(
        \mathsf G_2
    \). Then
    \(
        \elem_G
    \) is perfect, i.e.
    \(
        \elem_G(R) = [\elem_G(R), \elem_G(R)]
    \) for any
    \(
        R \in \Ring_K
    \).
\end{theorem}
\begin{proof}
    By \cite[lemma 5]{luzgarev-stavrova} for every prime ideal
    \(
        \mathfrak p \leqt K
    \) there are
    \(
        s \in K \setminus \mathfrak p
    \), an isotropic pinning
    \(
        (T, \Psi)
    \) of \(G_s\) of rank \(\geq 2\), and
    \(
        n \geq 1
    \) such that
    \[
        \elem_{G, T, \Psi}(x^n K_s[x, y])
        \leq [
            \elem_{G, T, \Psi}(x K_s[x, y]),
            \elem_{G, T, \Psi}(x K_s[x, y])
        ].
    \]
    Thus by lemma \ref{power-idem} for every
    \(
        \alpha \in \Psi
    \) and basis vector
    \(
        e_i \in P_\alpha
    \) we have
    \[
        t_\alpha(e_i \cdot \mathcal R^{(s^\infty)})
        \leq [
            \elem_{G, T, \Psi}(
                \mathcal R^{(s^\infty)}
            ),
            \elem_{G, T, \Psi}(
                \mathcal R^{(s^\infty)}
            )
        ].
    \]
    This implies the claim.
\end{proof}

\bibliographystyle{plain}
\bibliography{references}

\end{document}